\documentclass[twoside,10pt]{article}

\usepackage[pagewise]{lineno}%\linenumbers

\usepackage[colorlinks,
            citecolor=red]{hyperref}

\usepackage{makecell}
\usepackage{multirow}
\usepackage[noadjust]{cite}%remove space before cite

\usepackage[english]{babel}
\usepackage{amsmath}
\usepackage{amsthm}
\usepackage{amsfonts}
\usepackage{amssymb}
\usepackage{graphicx}
\usepackage{colortbl,dcolumn}
\usepackage{paralist}  % generate "compactenum"
\usepackage{latexsym}
\usepackage{amsfonts}
\usepackage{amssymb}
\usepackage{cite}
\usepackage{appendix}
\usepackage{subcaption}
\usepackage{tikz}

     \newtheorem{lemma}{\bf Lemma}[section]
     \newtheorem{theorem}{\bf Theorem}[section]

     \newtheorem{definition}{\bf Definition}[section]

     \numberwithin{equation}{section}

\begin{document}
\title{{\LARGE Effective Boundary Conditions for Heat Equation Arising from Anisotropic and Optimally Aligned Coatings in Three Dimensions}
 \footnotetext{
 E-mail addresses: gengxingri@u.nus.edu.\\ }}

\author{{Xingri Geng $^{a, b}$}\\[2mm]
\small\it{ $ ^a $Department of Mathematics}, Southern University of Science and Technology, Shenzhen, P.R. China\\
\small\it $^b$ Department of Mathematics, National University of Singapore,  Singapore 
}

\date{}

\maketitle

\begin{abstract}

We discuss the initial boundary value problem for a heat equation in a domain surrounded by a layer. The main features of this problem are twofold: on one hand, the layer is thin compared to the scale of the domain, and on the other hand, the thermal conductivity of the layer is drastically different from that of the bulk; moreover, the bulk is isotropic, but the layer is anisotropic and ``optimally aligned" in the sense that any vector in the layer normal to the interface is an eigenvector of the thermal tensor. We study the effects of the layer by thinking of it as a thickless surface, on which ``effective boundary conditions" (EBCs) are imposed. In the three-dimensional case, we obtain EBCs by investigating the limiting solution of the initial boundary value problem subject to either Dirichlet or Neumann boundary conditions as the thickness of the layer shrinks to zero. These EBCs contain not only the standard boundary conditions but also some nonlocal ones, including the Dirichlet-to-Neumann mapping and the fractional Laplacian. One of the main features of this work is to allow the drastic difference in the thermal conductivity in the normal direction and two tangential directions within the layer.
\end{abstract}

\medskip

\noindent{\bf Keywords.} heat equation, thin layer, energy estimates, asymptotic behavior, effective boundary conditions.

\medskip
\noindent{\bf AMS subject classifications.} 35K05, 35B40, 35B45,74K35.

%\tableofcontents

\section{Introduction}

This paper is concerned with the scenario of insulating an isotropic conducting body with a coating whose thermal conductivity is anisotropic and drastically different from that of the body. Moreover, the coating is thin compared to the scale of the body, resulting in multi-scales in the spatial variable. The difference in thermal conductivity and spatial size leads to computational difficulty. Some examples of this type of situation include cells with their membranes and thermal barrier coatings (TBCs) for turbine engine blades (see Figure \ref{fig}). To handle such situations, we view the coating as a thickless surface as its thickness shrinks to zero,  on which ``effective boundary conditions" (EBCs) are imposed. These EBCs not only provide an alternative way for numerical computation but also give us an analytic interpretation of the effects of the coating.

% In the three-dimensional case, we derive EBCs by investigating the limiting solution of the initial boundary value problem subject to either Dirichlet or Neumann boundary conditions.  obtain corresponding ``effective boundary conditions"(EBCs), which will also reveal the physical effects of the thin layer analytically. 

The main purpose of this work is to find effective boundary conditions rigorously in a three-dimensional domain. In the article of Chen, Pond, and Wang \cite{CPW},  EBCs were studied in the two-dimensional case when the coating is anisotropic and ``optimally aligned". However, it is not straightforward to extend their results in three dimensions because a degenerate equation that never happens in two dimensions arises. This paper treats the case when the domain is three-dimensional, and the coating is ``optimally aligned" with two tangent diffusion rates that may be different, which has not been covered by the previous results yet.

\begin{center}
\begin{tikzpicture}
\def\angle{60}%
\pgfmathsetlengthmacro{\xoff}{2cm*cos(\angle)}%
\pgfmathsetlengthmacro{\yoff}{1cm*sin(\angle)}%
\draw (0,0) circle[x radius=3.5cm, y radius=2cm] ++(1.7*\xoff,1.7*\yoff) node{$\Omega_2$} ++(1.2*\xoff,-0.2*\yoff)  node{$\partial\Omega$};
\draw (0,0) circle[x radius=3cm, y radius=1.5cm] node{$\Omega_1$}
++(0*\xoff,-3*\yoff) node{Figure 1: $\Omega=\overline{\Omega}_1\cup\Omega_2.$} ++(2*\xoff,1.7*\yoff)  node{$\Gamma$};

\fill (0,1.5) circle [radius = 1pt];
\draw ++(0*\xoff,1.5*\yoff) node{p};

\draw[<-] (3,0)--(3.25,0) node{$\delta$};
\draw[->] (3.3,0)--(3.5,0);
\draw[->] (0,1.5)--(0,1.8);
\draw ++(-0.2*\xoff,2*\yoff) node{\textbf{n}};
%\draw[->] (0,1.5)--(0.4,1.5);
%\draw ++(0.6*\xoff,1.8*\yoff) node{\textbf{s}};
\end{tikzpicture}\label{fig}
\end{center}

To be more specific, we introduce our mathematical model as follows: let the body $\Omega_1$ be surrounded by the coating $\Omega_2$ with uniform thickness $\delta > 0$; let the domain $\Omega=\overline{\Omega}_1\cup\Omega_2\subset \mathbb{R}^3$ as shown in Figure \ref{fig}.  For any finite $T>0$,  consider the initial boundary value problem with the Dirichlet boundary condition
\begin{equation*} \label{PDE1}
    \left\{
             \begin{array}{llr}
             u_t-\nabla \cdot (A(x)\nabla u)=f(x,t), &\mbox{$(x,t)\in Q_T,$}\\
             u=0,     &\mbox{$(x,t)\in S_T,$}&  \tag{1.1}\\
             u=u_0, &\mbox{$(x,t)\in\Omega\times\{0\},$}   
             \end{array}
  \right.
\end{equation*}
where  $Q_T:=\Omega\times(0,T)$  and  $S_T:=\partial\Omega\times(0,T)$. Suppose that $u_0 \in L^2(\Omega), f \in L^2(Q_T)$, and $A(x)$ is the thermal conductivity given by
\begin{equation*}
    A(x)=\left\{
        \begin{aligned}
              &kI_{3\times3}, &x\in \Omega_1,\\
              &\left(a_{ij}(x)\right)_{3\times3}, & x\in\Omega_2,
        \end{aligned}
    \right.
\end{equation*}
where $k$ is a positive constant independent of $\delta> 0$, and the positive-definite matrix $ (a_{ij} (x))$ is anisotropic and ``optimally aligned'' in the coating $\Omega_2$, which means that any vector inside the coating normal to the interface is always an eigenvector of $A(x) -$ see (\ref{OAC}) below for the precise definition.

Moreover, we also consider the initial value problem with the Neumann boundary condition
\begin{equation*}\label{PDE2}
    \left\{
             \begin{array}{llr}
             u_t-\nabla \cdot (A(x)\nabla u)=f(x,t), &\mbox{$(x, t)\in Q_T,$}\\
             \frac{\partial u}{\partial \textbf{n}_A}=0,     &\mbox{$(x,t)\in S_T,$}& \\
             u=u_0, &\mbox{$(x,t)\in\Omega\times\{0\},$}   
             \end{array}\tag{1.2}
  \right.
 \end{equation*}
where $\textbf{n}_A$ is the co-normal vector $A(x)\textbf{n}$, with $\textbf{n}$ being the unit outer normal vector field on $\Gamma(=\partial\Omega_1)$. In this case, the Neumann boundary condition is the same as $\frac{\partial u}{\partial \textbf{n}}=0$ since the coating is ``optimally aligned'' $-$ see below. 

\smallskip
We introduce the curvilinear coordinates $(p, r)$ by defining mapping
\begin{equation*}\label{curvilinear}
\begin{split}
F:  \quad & \Gamma \times (0,  \delta)  \longrightarrow   \mathbb{R}^3, \\
	 & F(p ,r) = p + r \textbf{n}(p) \in \Omega_2,
\end{split}
\end{equation*}
where $p$ is the projection of $x$ on $\Gamma$; $\textbf{n}(p)$ is the unit normal vector of $\Gamma$ pointing out of $\Omega_1$ at $p$; $r$ is the distance from $x$ to $\Gamma$. In the three-dimensional case, since the thermal tensor $A(x)$ is positive-definite, it has three orthogonal eigenvectors and corresponding eigenvalues. Every eigenvalue measures the thermal conductivity of the coating in the corresponding direction. By saying the coating $\Omega_2$ is optimally aligned, we mean that 
\begin{equation*}\label{OAC}
     A(x)\textbf{n}(p)=\sigma \textbf{n}(p), \quad \forall x\in \Omega_2, \tag{1.3}
\end{equation*}
where $\Omega_2$ is thin enough and $\Gamma$ is smooth enough such that the projection $p$ of $x$ onto $\Gamma$ is unique. % and n(p)\textbf{n}(p) is the unit outer normal vector of Γ\Gamma at pp. 
This concept was first introduced by Rosencrans and Wang \cite{RW} in $2006$.

Because of the optimally aligned coatings, $A(x)$ must have two eigenvectors in the tangent directions. If  $ A(x)$ has two identical eigenvalues in the tangent directions, then within the coating $\Omega_2$, we assume that the thermal tensor $A(x)$ satisfies 
\begin{equation*}\label{case1}
\textit{Type I }  \text{ condition}: \quad  A(x)\textbf{s}(p)=\mu \textbf{s}(p), \quad \forall x\in \Omega_2, \tag{1.4}
\end{equation*}
where $\textbf{s}(p)$ is an arbitrary unit tangent vector of $\Gamma$ at $p$; $\sigma$ and $\mu$ are called the normal conductivity and the tangent conductivity, respectively.

If $A(x)$ has two different eigenvalues $\mu_1$ and $\mu_2$ in the tangent directions, then two tangent directions are fixed on $\Gamma$. According to the Hairy Ball Theorem in algebraic topology, there is no nonvanishing continuous tangent vector field on even-dimensional $n-$spheres. Therefore, in this paper, we consider $\Gamma$ to be a topological torus that is any topological space homeomorphic to a torus. Within the coating $\Omega_2$, we assume that the thermal tensor $A(x)$ satisfies 
\begin{equation*}\label{case2}
\textit{Type II } \text{ condition:} \quad
    A(x)\pmb{\tau}_1(p)=\mu_1 \pmb{\tau}_1(p), \quad 
    A(x)\pmb{\tau}_2(p)=\mu_2 \pmb{\tau}_2(p), \tag{1.5}
\end{equation*}
where $\pmb{\tau}_1(p)$ and $\pmb{\tau}_2(p)$ are two orthonormal eigenvectors of $A(x)$ in the tangent plane of $\Gamma$ at $p$; $\mu_1$ and $\mu_2$ are two different tangent conductivities in the corresponding tangent directions.

Throughout this article, $\Omega_1$ is fixed and bounded with $C^2$ smooth boundary $\Gamma$; the coating $\Omega_2$ is uniformly thick with $\partial\Omega$ approaching $\Gamma$ as $\delta \to 0$; $(\sigma, \mu, \mu_1, \mu_2) = \left(\sigma(\delta), \mu(\delta), \mu_1(\delta), \mu_2(\delta) \right)$ are positive functions of $\delta$. 

\smallskip
There have been rich, deep, and interesting results about the idea of using EBCs in the literature. It can date back to the classic book of Carslaw and Jaeger \cite{HJ}, where EBCs were first recorded. Subsequently, Sanchez-Palencia \cite{SP} first investigated the ``interior reinforcement problem" for the elliptic and parabolic equations in a particular case when the reinforcing material is lens-shaped. Following this line of thought, Brezis, Caffarelli, and Friedman \cite{BCF} rigorously studied the elliptic problem for both interior and boundary reinforcement. See Li and Zhang \cite{L, LZ} for further development. For the case of a rapid oscillating thickness of the coating, see \cite{BK}. Later on, lots of follow-up works of EBCs for general coatings and ``optimally aligned coatings" emerged (see \cite{ LZ, LW2017, LW, CPW, LZRW2009, LWZZ, GENG, LSWW2021, LW2022}). Furthermore, there is also a review paper \cite{XW} that provides a thorough investigation of this topic.

\smallskip
The layout of this paper is as follows. Section \ref{sec: estimates} is devoted to establishing some basic energy estimates and a compactness argument, showing that $u$ converges to some $v$ after passing to a subsequence of $\{u\}_{\delta>0}$ as $\delta \to 0$. In Section \ref{sec: case1}, we derive effective boundary conditions on $\Gamma \times (0, T)$ for the case of \textit{Type I} condition, in which two auxiliary functions are developed via harmonic extensions. In Section \ref{sec: case2}, based on two different harmonic extensions, we address effective boundary conditions on $\Gamma \times (0, T)$  for the case of \textit{Type II} condition.

%----------------------Weak solutions------------------------%
\section{Weak solutions}\label{sec: estimates}
In this section we begin with some \textit{a priori} estimates, by which a compact argument is established to study the asymptotic behavior of the weak solution of (\ref{PDE1}) or (\ref{PDE2}).
\subsection{Preliminaries}
Before going into energy estimates, we first introduce some important Sobolev spaces: let $W^{1,0}_2(Q_T)$ be the subspace of functions belonging to $L^2(Q_T)$ with first order weak derivatives in $x$ also being in $L^2(Q_T)$; $W^{1,1}_2(Q_T)$ is defined similarly with  the first order weak derivative in $t$ belonging to $L^2(Q_T)$; $W^{1,0}_{2,0}(Q_T)$ is the closure in $W^{1,0}_2(Q_T)$ of $C^\infty$ functions vanishing near $\overline {S}_T$, and $W^{1,1}_{2,0}(Q_T)$ is defined similarly. Furthermore, denote
$V^{1,0}_{2,0}(Q_T):=W^{1,0}_{2,0}(Q_T)\cap C  \left([0,T];L^2(\Omega)\right).$

Let us define one more Sobolev space on $Q^1_T= \Omega_1\times(0,  T): V^{1,0}_{2}(Q^1_T)=W^{1,0}_{2}(Q^1_T)\cap C  \left([0,T];L^2(\Omega_1)\right).$
We endow all these spaces with natural norms.

For simplicity, we write $\int_{Q_T} u(x, t) dxdt$ instead of $\int_0^T\int_{\Omega} u(x, t) dxdt$.
\begin{definition}\label{def}
A function $u$ is said to be a weak solution of the Dirichlet problem (\ref{PDE1}), if $u\in V^{1,0}_{2,0}(Q_T)$ and for any $\xi \in C^\infty (\overline{Q}_T) $ satisfying $\xi = 0$ at $t=T$ and also near $S_T$, it holds that
 \begin{equation*}
    \mathcal{A}[u,\xi] := -\int_{\Omega} u_0 \xi(x,0)dx + \int_{Q_T} \left( A(x)\nabla u \cdot \nabla\xi - u \xi_t-f\xi \right)dxdt=0.  \tag{2.1}
 \end{equation*}
\end{definition}
The weak solution of the Neumann problem (\ref{PDE2}) is defined in the same way, except that $u\in  V^{1,0}_{2}(Q_T)$, and $\xi \in C^\infty (\overline{Q}_T ) $ satisfies $\xi = 0$ at $t=T$. Moreover, for any small $\delta > 0$, (\ref{PDE1}) or (\ref{PDE2})  admits a unique weak solution $u \in W^{1,0}_{2}(Q_T)\cap C  \left([0, T];L^2(\Omega)\right)$. As is well known, $u$ satisfies the following ``transmission conditions" in the weak sense
\begin{equation*}\label{trans}
  u_1=u_2, \quad k\nabla u_1 \cdot \textbf{n} = \sigma \nabla u_2 \cdot \textbf{n} \quad \text{ on }\Gamma,  \tag{2.2}
\end{equation*}
where $u_1$ and $u_2$ are the restrictions of $u$ on $\Omega_1\times (0, T)$ and $\Omega_2\times (0, T)$, respectively.

%-----------------Energy estimates----------------------%
\subsection{Basic energy estimates}
In the sequel, for notational convenience, let $C(T)$ represent a generic positive constant depending only on $T$; let $O(1)$ represent a quantity that varies from line to line but is independent of $\delta$. We provide the following energy estimates for the weak solution of (\ref{PDE1}) or (\ref{PDE2}). 
\begin{lemma}\label{Estimate1}
Suppose $f \in L^2(Q_T)$ and $u_0 \in L^2(\Omega)$. Then, any weak solution $u$ of  (\ref{PDE1}) or (\ref{PDE2}) satisfies the following inequalities.
\begin{equation*}
 \begin{split}
     \text{(i)}&\max_{t\in[0,T]}\int_{\Omega}u^2(x,t)dx+\int_{Q_T}\nabla u \cdot A(x)\nabla u dxdt \leq C(T) \left(\int_{\Omega}u_0^2dx +\int_{Q_T}f^2dxdt \right),\\
     \text{(ii)}& \max_{t\in[0,T]} t\int_{\Omega}\nabla u \cdot A(x)\nabla udx+\int_{Q_T}t u_t^2dxdt\leq C(T)\left(\int_{\Omega}u_0^2dx +\int_{Q_T}f^2dxdt\right).
 \end{split}
\end{equation*}
\end{lemma}
\begin{proof}
$(i)$ and $(ii)$ can be proved formally by a standard method. Multiplying (\ref{PDE1}) or (\ref{PDE2}) by $u$ and $t u_t$ respectively, we perform the integration by parts in both $x$ and $t$ over $\Omega \times (0, T)$. By the same analysis on the Galerkin approximation of $u$, this formal argument can be made rigorous. Hence, we omit the details.
\end{proof}
We prove our results using only $H^1$ a priori estimates, and higher order estimates are not needed for Theorem \ref{thm} and \ref{thm2} here. We refer interested readers to \cite[Theorem 5]{CPW} for more general higher order estimates for (\ref{PDE1}) or (\ref{PDE2}). 

For even general coefficients $A = A( x,t) =  \left(a_{ij}(x,t) \right)_{N \times N}$, let $a_{ij}(x,t)$ satisfy 
\begin{equation*}
\underset{i,j}{\sum} a_{ij}(x,t) \xi_i \xi_j  \geq  \lambda_0 |\xi|^2,
\end{equation*}
for any $\xi \in \mathbb{R}^N$ and some constant $\lambda_0>0$.  We also address the regularity results of $u$ near the interface $\Gamma$ without rigorous proof.
\begin{theorem}
Let $m$ be an integer with $m \geq 2$ and $ a \in (0, 1)$. Suppose that $\Gamma \in C^{m+a}$, $f \in C^{m-2+a, (m-2+a)/2}\left( \overline{\Omega}_h\times[0, T] \right)(h=1,2)$, and $a_{ij}\in C^{m-1+a,(m-1+a)/2}(\overline{\Omega}_h\times[0,T])$, then for any $t_0>0$, the weak solution $u$ of (\ref{PDE1}) or (\ref{PDE2}) satisfies
\begin{equation*}
    u\in C^{m+a,(m+a)/2}(\overline{\mathcal{N}}_h\times [t_0,T]),
\end{equation*}
where $\mathcal{N}$ is a narrow neighborhood of $\Gamma$ and $\mathcal{N}_h=\mathcal{N}\cap \Omega_h.$
\end{theorem}
\begin{proof}
The  proof of the theorem can be found in \cite{CPW} by using the idea of Nirenberg in \cite{LV}.
\end{proof}
%---------------Compact argument------------------%
\subsection{A compactness argument} 
We next turn to the compactness of the family of functions $\{u\}_{\delta>0}.$
\begin{theorem} \label{cpt}
Suppose that $\Gamma \in C^2$, $u_0\in L^2(\Omega)$ and $f\in L^2(Q_T)$ with all functions remaining unchanged as $\delta\to 0.$  Then, after passing to a subsequence of $\delta \to 0$, the weak solution $u$ of (\ref{PDE1}) or (\ref{PDE2}) converges to some $v$ weakly in $W^{1,0}_{2}(\Omega_1 \times(0,T))$, strongly in $C\left([0,T]; L^2(\Omega_1) \right)$. 
\end{theorem}

\begin{proof}[Proof of the theorem \ref{cpt}]
The proof this theorem is similar to that of \cite[Proposition 3.1]{CPW}, and hence we omitted the details.
\end{proof}
\section{EBCs for \textit{Type I} condition }\label{sec: case1}
Throughout this section, we always have the assumption of  \textit{Type I} condition (\ref{case1}). Under this condition, we aim to derive EBCs on $\Gamma \times (0, T)$ as the thickness of the layer shrinks to zero. 
\begin{theorem}\label{thm}
Suppose that $A(x)$ is given in (\ref{PDE1}) or (\ref{PDE2}) and satisfies (\ref{case1}). Let $u_0 \in L^2(\Omega)$ and $f \in L^2(Q_T)$ with functions being independent of $\delta$. Assume further that $\sigma$ and $\mu$ satisfy the scaling relationships
\begin{equation*}
    \lim_{\delta\to 0}\sigma\mu=\gamma\in[0,\infty], \quad
\lim_{\delta\to 0}\frac{\sigma}{\delta}=\alpha\in[0,\infty], \quad
\lim_{\delta\to 0}\mu\delta=\beta\in[0,\infty].
\end{equation*}
Let $u$ be the weak solution of (\ref{PDE1}) or  (\ref{PDE2}), 
then as $\delta \to 0$, $u\to v$ weakly in $W_2^{1,0}(\Omega_1 \times (0,T))$, strongly in $C([0,T];L^2(\Omega_1))$, where $v$ is the weak solution of 
\begin{equation*} \label{EPDE}
    \left\{
             \begin{array}{ll}
             v_t- k\Delta v=f(x,t), & (x,t)\in \Omega_1\times(0,T),  \\  
             v = u_0, & (x, t) \in \Omega_1\times \{0\},
             \end{array}
\right. \tag{3.1}
\end{equation*}
subject to the effective boundary conditions on $\Gamma\times (0,T)$ listed in Table \ref{tb1}.
\end{theorem}

\begin{table}[!htbp]
    \centering
    \caption{Effective boundary conditions on $\Gamma \times (0, T)$.}\label{tb1}
    
    EBCs on $\Gamma\times (0,T)$ for (\ref{PDE1}).\\
    
    \medskip
    
    \begin{tabular}{l|llc}
        
         As $\delta\to 0$   
         &  \quad $\frac{\sigma}{\delta}\to 0$ 
         & $\frac{\sigma}{\delta}\to\alpha\in(0,\infty)$ 
         &  \quad $\frac{\sigma}{\delta}\to\infty$\\
        \hline
        \hline
         $\sigma\mu\to 0$  
         & \quad $\frac{\partial v}{\partial \textbf{n}}=0$	
         & $k\frac{\partial v}{\partial \textbf{n}}=-\alpha v$	
         &  $v=0$\\
        \hline
         $\sqrt{\sigma\mu}\to\gamma\in (0,\infty)$	
         & \ $k\frac{\partial v}{\partial \textbf{n}}=\gamma \mathcal{J}_D^\infty [v]$ 
         &  $k\frac{\partial v}{\partial \textbf{n}}=\gamma \mathcal{J}_D^{\gamma/\alpha} [v]$	
         & $v=0$\\
        \hline
         $\sigma\mu\to \infty$
         & \makecell{ $\nabla_{\Gamma} v =0$,\\ $\int_{\Gamma}\frac{\partial v}{\partial\textbf{n}}=0 $ 
        }	
        & \makecell{$\nabla_{\Gamma} v =0$,\\ $\int_{\Gamma}(k\frac{\partial v}{\partial\textbf{n}}+\alpha v) =0 $
        }	
        &  $v=0 $  \\
        \hline
    \end{tabular}

    \bigskip
    
    EBCs on $\Gamma\times (0,T)$ for (\ref{PDE2}).\\
    
    \medskip
   
    \begin{tabular}{l|llc}
    
      As $\delta\to 0$ 
      &  $\mu\delta\to 0$ 
      & \quad $\mu\delta\to\beta\in(0,\infty)$ 
      &  \quad $\mu\delta\to \infty$\\
      \hline
      \hline
     $\sigma\mu\to 0$ 	
      &  $\frac{\partial v}{\partial \textbf{n}}=0$	
      &   \qquad $\frac{\partial v}{\partial \textbf{n}}=0$	
      & \quad $\frac{\partial v}{\partial \textbf{n}}=0$\\
        \hline
       $\sqrt{\sigma\mu}\to\gamma\in (0,\infty)$	
       &   $\frac{\partial v}{\partial \textbf{n}}=0$	
       &  \quad $k\frac{\partial v}{\partial \textbf{n}}=\gamma \mathcal{J}_N^{\beta/\gamma} [v]$	
       &  \quad $k\frac{\partial v}{\partial \textbf{n}}=\gamma \mathcal{J}_N^\infty [v]$\\
       \hline
        $\sigma\mu\to \infty$	
       &  $\frac{\partial v}{\partial \textbf{n}}=0$	
       & \quad $k\frac{\partial v}{\partial\textbf{n}}= \beta\Delta_{\Gamma}v$ 
       & \quad \makecell{$\nabla_{\Gamma} v =0$,\\$\int_{\Gamma}\frac{\partial v}{\partial \textbf{n}}=0 $}\\
        \hline
    \end{tabular}
\end{table}   
We now focus on the boundary conditions arising in Table \ref{tb1}. The boundary condition $\nabla_\Gamma v =0$ on $\Gamma \times (0, T)$ indicates that $v$ is a constant in the spatial variable (but it may depend on $t$), where $\nabla_\Gamma$ is the surface gradient on $\Gamma$. The operator $\Delta_\Gamma $ is the Laplacian-Beltrami operator defined on $\Gamma$, and the boundary condition
$k\frac{\partial v}{\partial \textbf{n}}=\beta \Delta_\Gamma v $ can be understood as a second-order partial differential equation on $\Gamma$, revealing that the thermal flux across $\Gamma$ in the outer normal direction causes heat accumulation that diffuses with the diffusion rate $\beta$.

$\mathcal{J}_D^{H}$ and $\mathcal{J}_N^{H}$, as shown in Table \ref{tb1}, are linear and symmetric operators mapping the Dirichlet value to the Neumann value. More precisely, for $ H \in(0, \infty)$, and smooth $g$ defined on $\Gamma$, we define
\begin{equation*}
    \mathcal{J}_D^H[g](s) := \Theta_R(s, 0)  \quad \text{ and } \quad \mathcal{J}_N^H[g](s):=\Pi_R(s, 0),
\end{equation*}
where $\Theta$ and $\Pi$ are, respectively, the bounded solutions of 
 \begin{equation*}
    \left\{
             \begin{array}{ll}
                   \Theta_{RR}+\Delta_\Gamma      \Theta=0 , & \Gamma\times(0,H),\\
                  \Theta (s, 0) = g(s), &      \Theta(s, H)=0,
             \end{array}
  \right.
\quad
 \left\{
             \begin{array}{ll}
                \Pi_{RR}+\Delta_\Gamma   \Pi=0 , & \Gamma\times(0,H),\\
               \Pi (s, 0)=g(s), &   \Pi_R(s, H)=0.
             \end{array}
  \right.
 \end{equation*}
The analytic formulas for $\mathcal{J}_D^H[g]$ and $\mathcal{J}_N^H[g]$ are given and deferred to Subsection \ref{sec3.2}. We then define
\begin{equation*}
    \left(\mathcal{J}_D^\infty[g], \mathcal{J}_N^\infty[g]\right) :=\underset{H \to \infty}{\lim}\left(\mathcal{J}_D^H[g], \mathcal{J}_N^H[g]\right),
\end{equation*}
where $\mathcal{J}_D^\infty[g] = \mathcal{J}_N^\infty[g] = - \left(-\Delta_\Gamma\right)^{1/2} g$ is the fractional Laplacian-Beltrami defined on $g$. 

\smallskip
We remark that the effective boundary conditions listed in Table \ref{tb1} are the same as those in \cite[Theorem 1]{CPW} where they call the operators $\mathcal{J}_D^\infty$ and $\mathcal{J}_N^\infty$ as the derivative of Hilbert transform rather than the fractional Laplacian. Notably,  our effective boundary conditions do not concern the time derivative, and we refer the interested reader to \cite{CR1990, GLR2022, LSWW2021} for the derivation of dynamic boundary conditions.

%---------------E&U of effective  model---------------------%
\subsection{Definition, existence and uniqueness of weak solutions of effective models}
We define weak solutions of (\ref{EPDE}) together with the boundary conditions in Table \ref{tb1}.
\begin{definition}\label{def1}
Let the test function $\xi\in C^\infty (\overline{Q^1_T}) $ satisfy $\xi=0$ at $t=T$.

$(1)$ A function $v$ is said to be a weak solution of (\ref{EPDE}) with the Dirichlet boundary condition $v=0$ if $v\in V^{1,0}_{2,0}(Q_T^1)$, and for any test function $\xi$, $v$ satisfies
\begin{equation*}\label{wksol1}
\begin{split}
    \mathcal{L}[v,\xi] :=& -\int_{\Omega_1}u_0(x)\xi(x,0)dx+\int_0^T\int_{\Omega_1} \left(k\nabla v \cdot \nabla\xi-v\xi_t-f\xi \right)dxdt
    =0.
\end{split}\tag{3.2}
\end{equation*}

$(2)$ A function $v$ is said to be a weak solution of (\ref{EPDE}) with the boundary conditions $\nabla_{\Gamma} v=0$ and
$\int_{\Gamma}(k\frac{\partial v}{\partial \textbf{n}}+\alpha v)=0$ for $\alpha \in [0, \infty)$ if for almost everywhere fixed $t\in (0, T)$, the trace of $v$ on $\Gamma$ is a constant, and if $\nabla_{\Gamma}\xi=0$ on $\Gamma$, it holds that $v\in V^{1,0}_2(Q_T^1)$ and $v$ satisfies
$$\mathcal{L}[v,\xi]=-\int_0^T\int_{\Gamma}\alpha v \xi dsdt.
$$

$(3)$ A function $v$ is said to be a weak solution of (\ref{EPDE}) with the boundary condition $k\frac{\partial v}{\partial \textbf{n}}=\mathcal{B}[v]$, where $\mathcal{B}[v]=-\alpha v$, or $\gamma \mathcal{J}_D^H[v]$, or $\gamma \mathcal{J}_N^H[v]$ for $H \in (0, \infty]$, if $v\in V^{1,0}_2(Q_T^1)$ and if for any test function $\xi$, $v$ satisfies
$$\mathcal{L}[v,\xi]=\int_0^T\int_{\Gamma} v \mathcal{B}[\xi] dsdt.
$$

$(4)$ A function $v$ is said to be a weak solution of (\ref{EPDE}) with the boundary condition $k\frac{\partial v}{\partial \textbf{n}}=\beta \Delta_{\Gamma} v$, if $v\in V^{1,0}_2(Q_T^1)$ with its trace belonging to $L^2\left((0,T); H^1(\Gamma)\right)$, and if for any test function $\xi$, $v$ satisfies
$$
\mathcal{L}[v,\xi]=-\beta\int_0^T\int_{\Gamma} \nabla_{\Gamma}v\nabla_{\Gamma}\xi dsdt.
$$

\end{definition}
A weak solution of (\ref{EPDE}) satisfies the initial value in the sense that $v(\cdot,t)\to u_0(\cdot)$ in $L^2(\Omega_1)$ as $t\to 0.$ Moreover, the existence and uniqueness of the weak solution of (\ref{EPDE}) with the boundary conditions in Tables \ref{tb1} are stated without proof in the following theorem. 
\begin{theorem}\label{Uni}
Suppose that $\Gamma \in C^1, u_0\in L^2(\Omega_1)$ and $f\in L^2(Q^1_T)$. Then, (\ref{EPDE}) with any boundary condition in Tables \ref{tb1} has one and only one weak solution as defined in Definition \ref{def1}.
\end{theorem}
\begin{proof}
For a rigorous proof of the theorem, the reader is referred to \cite{CPW} (see also \cite{LM1972} and \cite{W1987}).  
\end{proof}

Recall $\Gamma \in C^2$, and it is well-known (\cite[Lemma 14.16]{GT}) that for a small $\delta>0$, $ F$ is a $C^1$ smooth diffeomorphism from $\Gamma \times (0, \delta)$ to $\Omega_2$; $r = r(x)$ is a $C^2$ smooth function of $x$ and is seen as the inverse of the mapping $x = F(p, r)$. 
%Since \Gamma_1\Gamma_1 is C^2C^2 smooth, we parameterize the surface \Gamma_1\Gamma_1 by a finite number of local charts with standard compatibility conditions. 
By using local coordinates $s=(s_1, s_2)$ in a typical chart on $\Gamma$, we then have
\begin{equation*}\label{cur}
   p = p(s)=p(s_1,s_2),  \quad x = F(p(s), r) = F(s, r), \quad dx=(1+2Hr+\kappa r^2)dsdr  \quad \text{ in} \quad \overline{\Omega}_2,\tag{3.3}
\end{equation*}
where $ds$ represents the surface element; $H(s) $and $\kappa(s) $ are the mean curvature and Gaussian curvature at  $p$ on $\Gamma$, respectively. 

In the curvilinear coordinates $(s, r)$, the Riemannian metric tensor at $x \in \overline{\Omega}_2$ induced from $ \mathbb{R}^3$ is defined as $G( s,r)$ with elements 
 $$g_{ij}(s, r)=g_{ji}(s, r) = \langle F_{i}, F_{j} \rangle_{\mathbb{R}^3}, \quad i,j = 1,2,3,
 $$
where $F_i = F_{s_i}$ for $i =1,2$ and $F_3  = F_{r}$. Let $| G | :=  \det G$ and $g^{ij}(s,r)$ be the element of the inverse matrix of $G$, denoted by $G^{-1}$. 

In the curvilinear coordinates $(s, r)$, the derivatives of $u$ are given as follows
\begin{equation*}\label{derivative}
  \begin{split}
    \nabla u&=u_r \textbf{n} +\nabla_s u, \\
     \nabla_s u  = \sum_{i,j=1,2} g^{ij}(s,r) u_{s_j} & F_{s_i}(s,r) \quad \text{and} \quad \nabla_{\Gamma}u=\sum_{i,j=1,2} g^{ij}(s,0)u_{s_j}p_{s_i}(s),
\end{split} \tag{3.4}
\end{equation*}
\begin{equation*}
  \begin{split}
       \nabla \cdot \left(A(x)\nabla u \right) &= \frac{\sigma}{\sqrt{|G|}}\left(\sqrt{|G|} u_r\right)_r+\mu\Delta_{s}u, \\
       \Delta_{s}u=\nabla_s \cdot \nabla_s u &=\frac{1}{\sqrt{|G|}}\sum_{ij=1,2}\left(\sqrt{|G|}g^{ij}(s, r)u_{s_i}\right)_{s_j}. %\text{ and } \Delta_{\Gamma}u=\nabla_\Gamma \cdot \nabla_\Gamma u
  \end{split}\tag{3.5}
\end{equation*}
Moreover, if $A(x)$ satisfies \textit{Type I} condition (\ref{case1}), then in $\overline{\Omega}_2$, we have  
\begin{equation*}\label{eq36}
    A(x) = \sigma \textbf{n}(p)\otimes\textbf{n}(p) + \mu \sum_{ij}g^{ij}(s,r)F_{s_i}(s,r)\otimes F_{s_j}(s,r). \tag{3.6}
\end{equation*}

%-------------------Auxiliary Functions---------------------%
\subsection{Auxiliary functions }\label{sec3.2}
Our goal for this subsection is to construct two auxiliary functions and estimate their asymptotic behaviors when the thickness of the thin layer is sufficiently small. Our idea of developing these auxiliary functions is adapted from \cite{CPW} via a harmonic extension.  
 
We construct two auxiliary functions for \textit{Type I} condition (\ref{case1}) by defining $\theta$ and $\pi$. For every $t \in [0, T]$, let $\theta(s, r, t)$ and $\pi(s, r, t)$ be bounded solutions of 
\begin{equation*}\label{AF1}
    \left\{
             \begin{array}{ll}
             \sigma      \theta_{rr}+\mu \Delta_\Gamma      \theta=0 , & \Gamma \times (0, \delta),\\
             \theta (s, 0, t) = g(s), &      \theta(s, \delta, t)=0,
             \end{array}
  \right.
\quad
    \left\{
             \begin{array}{ll}
             \sigma   \pi_{rr}+\mu \Delta_\Gamma   \pi=0, & \Gamma \times (0, \delta),\\
               \pi(s, 0, t) = g(s), &  \pi_r(s, \delta, t)=0,
             \end{array}
  \right. \tag{3.7}
\end{equation*}
where $g(s):=g(p(s))=\xi(s,0, t)$. From the maximum principle, $\theta$ and $\pi$ are unique. 

\smallskip

Multiplying (\ref{AF1}) by $\theta$ and $\pi$ respectively, and implementing integration by parts over $\Gamma \times (0, \delta)$, we arrive at
\begin{equation*}\label{eq38}
    \begin{split}
      \int_0^\delta\int_{\Gamma} \left(\sigma     \theta_r^2+\mu |\nabla_\Gamma \theta|^2 \right) &= -\int_{\Gamma} \sigma  \theta_r(s,0,t)g(s),  \int_0^\delta \int_{\Gamma} \left(\sigma     \pi_r^2+\mu |\nabla_\Gamma \pi|^2 \right) = -\int_{\Gamma} \sigma \pi_r(s,0,t)g(s).
      \end{split}\tag{3.8}
\end{equation*} 
Multiplying (\ref{AF1}) by $u$ respectively and performing the integration by parts again, we get
\begin{equation*}\label{eqno39}
    \begin{split}
     \int_0^\delta \int_{\Gamma} \left(\sigma     \theta_r u_r+\mu \nabla_\Gamma \theta \cdot \nabla_\Gamma u \right)  &= -\int_{\Gamma} \sigma \theta_r(s,0,t) u(p(s), t),\\
     \int_0^\delta \int_{\Gamma} \left(\sigma     \pi_r u_r + \mu \nabla_\Gamma \pi \cdot \nabla_\Gamma u  \right)  &= -\int_{\Gamma} \sigma  \pi_r(s,0,t)u(p(s), t).
      \end{split}\tag{3.9}
\end{equation*} 

To eliminate $\sigma$ and $\mu$, we assert $r = R\sqrt{\sigma/\mu}$ and plug $r$ into (\ref{AF1}). Suppressing the time dependence, this leads to
$$\Theta(s, R)= \theta(s, R\sqrt{\sigma/\mu}, t), \quad \Pi(s, R) = \pi(s, R\sqrt{\sigma/\mu}, t).
$$ 
Consequently, (\ref{AF1}) is equivalent to
 \begin{equation*}\label{rescale1}
    \left\{
             \begin{array}{ll}
                   \Theta_{RR}+\Delta_\Gamma \Theta = 0, & \Gamma \times (0, h),\\
                  \Theta (s, 0) = g(s), & \Theta(s, h)=0,
             \end{array}
  \right.
\quad
  \left\{
             \begin{array}{ll}
                \Pi_{RR}+\Delta_\Gamma  \Pi = 0 , & \Gamma \times (0, h),\\
               \Pi(s, 0) = g(s), & \Pi_R(s, h)=0,
             \end{array}
  \right.\tag{3.10}
\end{equation*}
where $h:=\delta\sqrt{\frac{\mu}{\sigma}}=\frac{\mu\delta}{\sqrt{\sigma\mu}}=\frac{\sqrt{\sigma\mu}}{\sigma/\delta}.$ We now define two Dirichlet-to-Neumann operators
\begin{equation*}\label{D2N}
\mathcal{J}_D^h[g](s) := \Theta_R(s,0) \quad \text{ and } \quad  \mathcal{J}_N^h[g](s) := \Pi_R(s,0). \tag{3.11}
\end{equation*}
Observe 
\begin{equation*}\label{eq312}
\begin{split}
    \sigma \theta_r(s, 0, t) &= \sqrt{\sigma\mu}      \Theta_R(s,0) = \sqrt{\sigma\mu} \mathcal{J}_D^h[g](s), \quad
    \sigma \pi_r(s, 0, t) = \sqrt{\sigma\mu}   \Pi_R(s, 0)=\sqrt{\sigma\mu}\mathcal{J}_N^h[g](s). 
\end{split}\tag{3.12}
\end{equation*}
Rigorous formulas for $\mathcal{J}_D^h[g]$ and $\mathcal{J}_N^h[g]$ are given in eigenvalues and eigenfunctions of $- \Delta_\Gamma$ by using separation of variables, from which it follows that
\begin{equation*}\label{eq313}
	\begin{split}
             \Theta(s, R) & = \sum_{n=1}^\infty  \frac{-g_{n}e^{-\sqrt{\lambda_n} h}}{2 \sinh(\sqrt{\lambda_n} h)} \left(e^{\sqrt{\lambda_n}R}-e^{\sqrt{\lambda_n}(2h-R)}\right)e_n(s),
    \end{split}\tag{3.13}
\end{equation*}
\begin{equation*}\label{eq314}
	\begin{split}
             \Pi(s, R) &=\sum_{n=1}^\infty  \frac{g_{n}e^{-\sqrt{\lambda_n} h}}{2 \cosh(\sqrt{\lambda_n} h)} \left(e^{\sqrt{\lambda_n}R} + e^{\sqrt{\lambda_n}(2h-R)}\right)e_n(s),
       \end{split}\tag{3.14}
\end{equation*}
where $ g_{n} := \langle e_n, g \rangle = \int_{\Gamma} e_n g ds$; $\lambda_n$ and $e_n(s)$ are  the  eigenvalues and the corresponding eigenfunctions of the Laplacian-Beltrami $-\Delta_{\Gamma}$ defined on $\Gamma$.

\smallskip
Subsequently, it follows from (\ref{D2N}) and  (\ref{eq313}) that
\begin{equation*}\label{eq315}
\begin{split}
  \mathcal{J}_D^h[g](s) = &-\sum_{n=1}^\infty  \frac{\sqrt{\lambda_n} e_n(s)g_n}{\tanh(\sqrt{\lambda_n}h)}, \quad
  \mathcal{J}_N^h[g](s)  = -\sum_{n=1}^\infty \sqrt{\lambda_n} e_n(s)g_n \tanh(\sqrt{\lambda_n}h). 
\end{split}\tag{3.15}
\end{equation*}
Furthermore, if $h\to H \in (0, \infty ]$, we have
\begin{equation*}\label{eq316}
          \begin{split}
           \left|\mathcal{J}_D^h[g](s)-\mathcal{J}_D^H[g](s)\right|&=\sum_{n=1}^\infty \sqrt{\lambda_n} e_n(s)g_n \left(\frac{1}{\tanh(\sqrt{\lambda_n}H)} -\frac{1}{\tanh(\sqrt{\lambda_n}h)} \right) \\
           &=|H-h|\sum_{n=1}^\infty \lambda_n e_n(s)g_n\frac{-4}{(e^{\sqrt{\lambda_n}h^\prime}-e^{-\sqrt{\lambda_n}h^\prime})^2}\\
           &= O(|H-h|),
          \end{split} \tag{3.16}
\end{equation*}
for some $h^\prime$ between $h$ and $H$. This implies the uniform convergence in $h$. By using a similar analysis as above, if $h\to H\in(0,\infty]$, $\mathcal{J}_N^h[g]$ converges uniformly to $\mathcal{J}_N^H[g]$ where $\mathcal{J}_D^\infty[g] = \mathcal{J}_N^\infty[g]  := -(-\Delta_\Gamma)^{1/2}g.$

\smallskip
In the follow-up, we are going to estimate the size of the term $\Theta_R(s, 0)$ and $\Pi_R(s, 0)$ for a sufficiently small $\delta$.  On one hand, if $h$ is small and $h \to 0$ as $\delta \to 0$, then it follows from (\ref{eq315}) that 
\begin{equation*}\label{eqno416}
\begin{split}
    \left|\Theta_R(s, 0)+\frac{g(s)}{h}\right|& \leq  h \|g\|_{C^2(\Gamma)}, \quad  \left|\Pi_{R}(s, 0) - h\Delta_\Gamma g\right| \leq O(h^3).
\end{split}\tag{3.17}
\end{equation*}
Combining this with (\ref{eq312}), we get
\begin{equation*}\label{eq318}
 \begin{split}
     \sqrt{\sigma \mu} \Theta_R(s,0) &= \frac{\sigma}{\delta}\left(-g(s)+O(h^2)\right), \quad
     \sqrt{\sigma\mu} \Pi_R(s,0) = \mu\delta \left(\Delta_\Gamma g(s) + O(h^2)\right).
 \end{split}\tag{3.18}
\end{equation*} 

On the other hand, if $h \to H \in (0, \infty]$ as $\delta\to 0$, then from the Taylor expansion for $\Theta(s, R)$, we obtain 
\begin{equation*}
         \Theta_R(s,0)=\frac{    \Theta(s, R)-     \Theta(s, 0)}{R}-\frac{R}{2}     \Theta_{RR}(s, \overline{R}), 
\end{equation*}
for some $\overline{R}\in [0,R]$. Taking $R=\min \{h,1\}$, from the maximum principle, we have
\begin{equation*}
\begin{split}
    \| \Theta_R(s, 0)\|_{L^\infty(\Gamma))}&\leq\frac{2}{R}\|     \Theta\|_{L^\infty(\Omega_2)}+R\|      \Theta_{RR}\|_{L^\infty(\Omega_2)}\leq \frac{3\|g\|_{C^2(\Gamma)}}{R},
\end{split}
\end{equation*}
from which it turns out that 
\begin{equation*}\label{eq319}
    \sqrt{\sigma\mu} \|\Theta_R\|_{L^\infty(\Gamma)}=\frac{O(1)\sqrt{\sigma\mu}}{R}. \tag{3.19}
\end{equation*}
By the similar analysis on $\Pi_R$, if $h \to H \in (0,\infty]$ as $\delta\to 0$, then we have
\begin{equation*}\label{eq320}
\begin{split}
    \|  \Pi_R \|_{L^\infty(\Gamma)}&= O(1).
\end{split}\tag{3.20}
\end{equation*}
 
We end this subsection by mentioning that for $H \in (0, \infty)$, $\mathcal{J}^H_D[g]$ and $\mathcal{J}^H_N[g]$ are defined for smooth $g$. However, it is easy to show that they are also well-defined for given any $g \in H^{\frac{1}{2}}(\Gamma)$ where 
 $H^{\frac{1}{2}}(\Gamma)$ is defined by the completion of smooth functions under the $H^{\frac{1}{2}}(\Gamma)$ norm. Moreover, $\mathcal{J}^H_D$ and  $\mathcal{J}^H_N: H^{\frac{1}{2}}(\Gamma) \to H^{-\frac{1}{2}}(\Gamma)$ are linear and symmetric, where $H^{-\frac{1}{2}}(\Gamma)$ is the dual space of $H^{\frac{1}{2}}(\Gamma)$.

%----------------------Main proof----------------------------%
\subsection{Proof of Theorem \ref{thm}}
The main result of this subsection is to prove Theorem \ref{thm}, in which we derive EBCs on $\Gamma \times (0, T)$.

\begin{proof}[Proof of Theorem \ref{thm}]
According to Theorem \ref{cpt}, the weak solution $u$ of (\ref{PDE1}) or (\ref{PDE2}) converges to some $v$ weakly in $W^{1,0}_{2} \left(\Omega_1\times (0, T) \right)$, and strongly in $C\left( [0, T]; L^2(\Omega_1) \right)$ after passing to a subsequence of $\delta > 0$. Thus, given any subsequence of $\delta$, we emphasize that we can ensure that $u \to v$ in all above spaces after passing to a further subsequence. In the further, we will show that $v$ is a weak solution of (\ref{EPDE}) with effective boundary conditions listed in Table \ref{tb1}. By what we have proved in Theorem \ref{Uni}, $v$ is unique. The fact that $u \to v$ without passing to any subsequence of $\delta >0$, is a consequence of the uniqueness.

\smallskip

To derive the EBCs on $\Gamma \times (0, T)$, we complete our proof in the following two steps: one is for the Dirichlet problem (\ref{PDE1}), and  the other is for the Neumann problem (\ref{PDE2}).

\medskip

\noindent \textbf{Step 1. Effective boundary conditions for the Dirichlet problem (\ref{PDE1}). }

To begin with the proof, we assume that all conditions in Theorem \ref{thm} hold. Let the test function $\xi \in C^{\infty}(\overline{\Omega}_1\times [0,T])$ with $\xi=0$ at $t=T$, and extend $\xi$ to the domain $\overline{\Omega}\times [0,T]$ by defining
\begin{equation*}
   \overline{\xi}(x, t)=\left\{
\begin{aligned}
&\xi(x, t), &x\in \overline{\Omega}_1, \\
&\theta(p(x), r(x), t), &x\in\Omega_2,
\end{aligned}
\right.
\end{equation*}
where $\theta $ is introduced in (\ref{AF1}). It is easy to check that $\overline{\xi}\in W^{1,1}_{2,0}(Q_T)$, and $\overline{\xi}$ is called the harmonic extension of $\xi$. 

\smallskip
Since $u$ is a weak solution of (\ref{PDE1}), it follows from Definition \ref{def} that
\begin{equation*}\label{eq51}
\begin{split}
    \mathcal{A}[u,\overline{\xi}]&=-\int_{\Omega}u_0(x)\overline{\xi}(x,0)dx+\int_{0}^{T}\int_{\Omega} \left(\nabla \overline{\xi}\cdot A \nabla u-u\overline{\xi}_t-f\overline{\xi} \right)dxdt=0.
\end{split}\tag{3.21}
\end{equation*}
Rewrite (\ref{eq51}) as 
\begin{equation*}\label{eq322}
\begin{split}
     \int_0^T\int_{\Omega_1} k\nabla\xi\cdot\nabla u dxdt-\int_\Omega u_0(x)\overline{\xi}(x,0)dx-\int_0^T\int_{\Omega} (u\overline{\xi_t}+f\overline{\xi}) dxdt
  =&-\int_0^T\int_{\Omega_2}\nabla \theta \cdot A\nabla u dxdt.
\end{split}\tag{3.22}
\end{equation*}
Since $u \to v $ weakly in $W^{1, 0}_2\left(\Omega_1 \times(0,T) \right)$, and strongly in $C\left([0, T]; L^2(\Omega_1)\right)$ as $\delta \to 0$, we summarize as
\begin{equation*}
\left\{
\begin{array}{ll}
     &  \int_{Q_T}u\xi_t dxdt \rightarrow \int_{Q^1_T}v\xi_t dxdt, \\
     & \int_{Q^1_T}\nabla u\cdot \nabla  \xi dxdt \to \int_{Q^1_T}\nabla v\cdot \nabla  \xi dxdt,\\
     &\int_{Q_T}f \overline{\xi} dxdt \to \int_{Q^1_T}f \xi dxdt,
\end{array}
\right. 
\end{equation*}
from which the left-hand side of (\ref{eq322}) is equivalent to
\begin{equation*}\label{eqno53}
    \begin{split}
      \mathcal{L}[v, \xi] :=&\int_0^T\int_{\Omega_1} k\nabla\xi\cdot\nabla v dxdt-\int_{\Omega_1}u_0(x)\xi(x,0)dx-\int_0^T\int_{\Omega_1} (v\xi_t+f\xi) dxdt. 
      \end{split} \tag{3.23}
\end{equation*}

The remainder of the following focuses on the right-hand side of (\ref{eq322}). Using the curvilinear coordinates $(s,r)$, by virtue of (\ref{cur}), (\ref{derivative}) and (\ref{eq36}), we have
\begin{equation*}\label{eq324}
\begin{split}
     RHS:=&-\int_0^T\int_{\Omega_2}\nabla \theta \cdot A\nabla u dxdt\\
     =&-\int_0^T\int_{\Gamma}\int_0^\delta \left(\sigma  \theta_r u_r+\mu\nabla_s \theta \nabla_s u \right) (1+2Hr+\kappa r^2) drdsdt \\
    =&-\int_0^T\int_{\Gamma}\int_0^\delta \left(\sigma \theta_r u_r+\mu\nabla_\Gamma \theta \nabla_\Gamma u \right) -\int_0^T\int_{\Gamma}\int_0^\delta (\sigma     \theta_r u_r+\mu\nabla_\Gamma \theta \nabla_\Gamma u)(2Hr+\kappa r^2)  \\
    & -\int_0^T\int_{\Gamma}\int_0^\delta \mu (\nabla_s \theta \nabla_s u-\nabla_\Gamma \theta \nabla_\Gamma u)(1+2Hr+\kappa r^2)\\
    =: & I+II+III.
\end{split} \tag{3.24}
\end{equation*}
Due to (\ref{eqno39}) and (\ref{eq312}), it holds that 
\begin{equation*}\label{eq325}
I := \int_0^T \mathcal{I} dt= \sqrt{\sigma\mu}\int_0^T\int_{\Gamma}u(p(s), t) \Theta_R(s, 0)dsdt. \tag{3.25}
\end{equation*}
Subsequently, in view of (\ref{eq38}) and (\ref{eq319}), it follows from  Lemma \ref{Estimate1} that
\begin{equation*}\label{eq56}
\begin{split}
    |II| \leq &\int_0^T \left|\int_{\Gamma}\int_0^\delta (\sigma     \theta_r u_r+\mu\nabla_\Gamma \theta \nabla_\Gamma u)(2Hr+\kappa r^2) drds\right|dt\\
    = & O(\delta) \int_0^T \left( \int_{\Gamma}\int_0^\delta \sigma \theta_r^2 +\mu |\nabla_\Gamma \theta |^2 \right)^{1/2}\left( \int_{\Omega_2} \sigma u_r^2+\mu |\nabla_\Gamma u|^2 \right)^{1/2}dt\\
    = & O(\delta) \int_0^T \frac{1}{\sqrt{t}}\left(\int_\Gamma \sigma |\theta_r(s,0,t)| \right)^{1/2}dt\\
    = & O(\delta)  \sqrt{T}(\sigma\mu)^{1/4} ||\Theta_R(s,0)||^{1/2}_{L^\infty(\Gamma)},
\end{split}\tag{3.26}
\end{equation*}
where we have used H\"oder inequality. Consequently, using (\ref{derivative}), (\ref{eq38}) and (\ref{eq319}), we have
\begin{equation*}\label{eq327}
\begin{split}
    |III|\leq & \left|\int_0^T\int_{\Gamma}\int_0^\delta \mu (\nabla_s     \theta \nabla_s u-\nabla_\Gamma \theta \nabla_\Gamma u)(1+2Hr+\kappa r^2) \right|\\
    = & O(\delta) \int_0^T\int_{\Gamma}\int_0^\delta \mu |\sum_{ij} \theta_{s_i} u_{s_j}|\\
    = & O(\delta) \int_0^T \left( \int_{\Gamma}\int_0^\delta \sigma \theta_r^2 +\mu |\nabla_\Gamma \theta |^2 \right)^{1/2}\left( \int_{\Omega_2} \sigma u_r^2+\mu |\nabla_\Gamma u|^2 \right)^{1/2}dt\\
    % \leq & O(\delta) \int_0^T \frac{1}{\sqrt{t}}\left(\int_\Gamma \sigma |\theta_r(s,0,t)|ds\right)^{1/2} dt\\
    = & O(\delta)  \sqrt{T}(\sigma\mu)^{1/4} ||\Theta_R(s, 0)||^{1/2}_{L^\infty(\Gamma)}, 
\end{split}\tag{3.27}
\end{equation*}
where Lemma \ref{Estimate1} and H\"oder inequality were used.

To investigate the asymptotic behavior of the right-hand side of \eqref{eq322} as $\delta \to 0$, we consider the following cases
$$(1)  \frac{\sigma}{\delta}\to 0, \quad (2) \frac{\sigma}{\delta}\to \alpha\in (0,\infty), \quad (3) \frac{\sigma}{\delta}\to \infty.
$$

\noindent
\textbf{Case $1$}. $\frac{\sigma}{\delta}\to 0$ as $\delta \to 0$. 

\noindent 
Subcase $(1i)$. $\sigma\mu \to 0$ as $\delta\to 0$. By Lemma \ref{Estimate1} and the trace theorem, it follows from \eqref{eq318}, \eqref{eq319} and \eqref{eq325} that
\begin{equation*}
	\begin{split}
		|I| = & O(1)\sqrt{\sigma\mu} \left\| \Theta_R(s,0) \right\|_{L^\infty(\Gamma)} \int_0^T ||u||_{H^1(\Omega_1)} dt
		=  O(1) T \max \{ \frac{\sigma}{\delta}, \sqrt{\sigma \mu} \},
	\end{split}
\end{equation*}
where we have used H\"older inequality. In view of \eqref{eq318}, \eqref{eq319}, \eqref{eq56} and \eqref{eq327}, we have
$$
| II + III | = O(\delta)  \sqrt{T} \max \{ \sqrt{\frac{\sigma}{\delta}}, (\sigma \mu)^{1/4} \}.
$$
Then,
\begin{equation*}
	\left| RHS \right| \to 0 \text{~ as ~} \delta \to 0,
\end{equation*}
from which we have $\mathcal{L}[v,\xi]=0$. Thus $v$ satisfies the boundary condition $\frac{\partial v}{\partial \textbf{n}}=0$ on $\Gamma \times (0, T)$.

\medskip
\noindent 
Subcase $(1ii)$. $\sqrt{\sigma\mu}\to \gamma \in (0, \infty)$ as $\delta\to 0$. In this case, $h\to \infty$. By the weak convergence of $u$, as $\delta\to 0$, it holds from \eqref{eq315} and \eqref{eq325} that
\begin{equation*}
    \begin{split}
      \mathcal{I} & = \sqrt{\sigma\mu}\int_{\Gamma} u\Theta_R(s,0)  \longrightarrow \gamma \int_{\Gamma} v\mathcal{J}_D^\infty[g].
    \end{split}
\end{equation*}
Moreover, combining (\ref{eq319}), (\ref{eq320}), (\ref{eq56}) and (\ref{eq327}), we have $|II+III| \to 0$ as $\delta\to 0$. It turns out that
$$ \mathcal{L}[v, \xi] = \gamma \int_0^T \int_{\Gamma}v\mathcal{J}_D^\infty[\xi],$$ which means that $v$ satisfies 
    $k\frac{\partial v}{\partial \textbf{n}}=\gamma \mathcal{J}_D^\infty[v]$ 
on $\Gamma \times (0,T)$.

\medskip
\noindent 
Subcase $(1iii)$. $\sigma\mu \to \infty$. In this case, $h \to \infty$ as $\delta\to 0$. Divided both sides of (\ref{eq51}) by  $\sqrt{\sigma\mu}$ and sending $\delta\to 0$, we obtain
\begin{equation*}
\int_0^T\int_{\Gamma}v \mathcal{J}_D^\infty[g] = 0.
\end{equation*}
Because the range of $\mathcal{J}_D^\infty[\cdot]$ contains $\{e_n\}_{n=1}^\infty$ for almost everywhere $t\in (0, T)$, it turns out that $\nabla_\Gamma v =0$ on $\Gamma.$ We further choose a special test function $\xi$ such that $\xi(s,0, t) = m(t)$ for some smooth function $m(t)$. Then, we construct a linear extension by defining $\theta(s, r, t)= (1-\frac{r}{\delta}) m(t)$.
Consequently, a direct computation leads to
\begin{equation*}\label{eq58}
\begin{split}
    RHS=&-\int_0^T\int_{\Omega_2}\nabla     \theta \cdot A\nabla u dxdt
    =\int_0^T \frac{\sigma m(t)}{\delta}\left(\int_0^\delta \int_{\Gamma} u_r (1+2Hr+\kappa r^2) \right)dt\\
    =&\int_0^T \frac{\sigma m(t)}{\delta}\left(\int_{\Gamma}u \right) dt-\int_0^T \frac{\sigma m(t)}{\delta}\int_0^\delta \int_{\Gamma} u (2H+2\kappa r)\\
   \leq &  \frac{\sigma }{\delta}\int_0^T m(t)\left( O(1) + O(\sqrt{\delta})\|u(\cdot,t)\|_{L^2(\Omega_2)}\right) dt,
\end{split}\tag{3.28}
\end{equation*} 
from which we derive $\mathcal{L}[v,\xi]=0$ as $\delta \to 0$. Then, $v$ satisfies $\int_{\Gamma}\frac{\partial v}{\partial \textbf{n}}=0$ on $\Gamma \times (0,T)$.

\medskip
\noindent
 \textbf{Case $2$}. $\frac{\sigma}{\delta}\to \alpha \in (0, \infty)$ as $\delta\to 0$. 

\noindent 
Subcase $(2i)$. $\sigma\mu \to 0$ as $\delta\to 0$. In this case,  $h \to 0$. From (\ref{eq318}) and (\ref{eq324})- (\ref{eq327}), we have 
\begin{equation*}
    I \longrightarrow -\alpha\int_0^T\int_{\Gamma}v\xi \quad \text{ and } \quad II+III \longrightarrow 0 \text{ as } \delta\to0,
\end{equation*}
from which it follows that
\begin{equation*}
    \mathcal{L}[v,\xi]=-\alpha\int_0^T\int_{\Gamma}v\xi. \tag{3.29}
\end{equation*}
So, $v $ satisfies $k\frac{\partial v}{\partial \textbf{n}}=-\alpha v$ on $\Gamma\times (0,T)$.

\smallskip
\noindent Subcase $(2ii)$. $\sqrt{\sigma\mu}\to \gamma \in (0, \infty)$ as $\delta\to 0$.
In this case, $h \to H = \gamma/\alpha \in (0, \infty)$.  By virtue of (\ref{eq318}) and (\ref{eq324})- (\ref{eq327}), it holds that 
\begin{equation*}
    I \longrightarrow \gamma\int_0^T\int_{\Gamma}v \mathcal{J}_D^{\gamma/\alpha}[\xi] \quad \text{ and } \quad II+III \longrightarrow 0 \text{  as  } \delta \to 0, 
\end{equation*}
from which we get $\mathcal{L}[v,\xi]=\gamma\int_0^T\int_{\Gamma}v\mathcal{J}_D^{\gamma/\alpha}[\xi].$
So, $v$ satisfies  $ k\frac{\partial v}{\partial \textbf{n}}=\gamma \mathcal{J}^{\gamma/\alpha}[v]$ on $\Gamma\times (0,T)$.

\medskip
\noindent Subcase $(2iii)$. $\sigma\mu \to \infty$ as $\delta\to 0$. In this case, $h \to \infty$. Divided both sides of (\ref{eq322}) by  $\sqrt{\sigma\mu}$ and sending $\delta\to 0$, we obtain $\int_0^T\int_{\Gamma}v\mathcal{J}_D^\infty[\xi]=0$, resulting in $\nabla_\Gamma v =0$ on $\Gamma.$
% \begin{equation*}
%     \frac{I}{\sqrt{\sigma \mu}} \to \int_0^T\int_{\Gamma} v \mathcal{J}_D^\infty[\xi].
% \end{equation*}
Using the same test function and the auxiliary function in Subcase $(1iii)$, we obtain $ \mathcal{L}[v,\xi]=-\alpha\int_0^T\int_{\Gamma}v\xi$ and $\nabla_\Gamma v = 0$ on  $\Gamma$, which means $v$ satisfies  $\int_{\Gamma}\left(k\frac{\partial v}{\partial \textbf{n}}+\alpha v \right)=0$
on $\Gamma\times(0,T)$.

\medskip
\noindent
\textbf{Case $3$}.  $\frac{\sigma}{\delta}\to \infty$ as $\delta\to 0$.

\noindent 
Subcase $(3i)$. $\sqrt{\sigma\mu}\to \gamma \in [0, \infty)$ as $\delta\to 0$. In this case,  $h \to 0$. Divided both sides of (\ref{eq322}) by $\sigma / \delta$ and sending $\delta \to 0$, a combination of (\ref{eq38}) and (\ref{eq324})- (\ref{eq327}) leads to
\begin{equation*}
 \frac{\delta}{\sigma} I \longrightarrow -\int_0^T\int_{\Gamma}v\xi = 0, 
\end{equation*}
from which $v$ satisfies $v=0$ on $\Gamma\times(0,T).$ 

\smallskip
\noindent Subcase $(3ii)$. $\sigma\mu \to \infty$ as $\delta\to 0$. In this case, after passing to a subsequence, we have $h\to H \in[0, \infty]$. If $H=0$, then divided both sides of (\ref{eq322}) by $\sigma / \delta$ and sending $\delta\to 0$, it yields  $v=0$ on $\Gamma\times (0, T)$. 

If $H\in (0, \infty]$, then divided both sides of (\ref{eq322}) by $\sqrt{\sigma\mu}$ and sending $\delta \to 0$, we have
\begin{equation*}
 \frac{I}{\sqrt{\sigma\mu}}\longrightarrow \int_0^T\int_{\Gamma}v \mathcal{J}_D^H[\xi]=0.
\end{equation*}
Employing the method analogous to that in Subcase $(1iii)$, for almost everywhere $t\in (0, T)$, we have $\nabla_\Gamma v = 0 \text{ and }\int_0^T\int_{\Gamma}vm(t)=0,$ which implies $v=0$ on $\Gamma\times(0,T)$.

\medskip
\noindent 
\textbf{Step 2. Effective boundary conditions for the Neumann problem (\ref{PDE2}). }

Let $\xi\in C^{\infty}(\overline{\Omega}_1\times [0,T])$ with $\xi=0$ at $t=T$ and extend the test function $\xi$ to $\overline{\Omega}\times [0,T]$ by defining
\begin{equation*}
  \overline{\xi}(x,t)=\left\{
       \begin{aligned}
    &\xi(x,t),& x\in \overline{\Omega}_1, \\
  &\pi(p(x), r(x), t),& x\in\Omega_2,
\end{aligned}
\right.
\end{equation*}
where $\pi$ is introduced in (\ref{AF1}). It is easy to see that $\overline{\xi}\in W^{1,1}_{2,0}(Q_T)$.

Thanks to the weak convergence of $\{u\}_{\delta >0}$, as $\delta \to 0$, it follows from Definition \ref{def1} that
\begin{equation*}\label{eq330}
    \mathcal{L}[u,\overline{\xi}]
      \longrightarrow \mathcal{L}[v, \xi] =-\underset{\delta \to 0}{\lim}\int_0^T\int_{\Omega_2}\nabla \pi \cdot A\nabla u dx dt. \tag{3.30}
\end{equation*}
In the following, we focus on the right-hand side of (\ref{eq330}). By using the curvilinear coordinates $(s, r)$ in (\ref{cur}), it can be rewritten as 
\begin{equation*}\label{eq511}
\begin{split}
     RHS:=&-\int_0^T\int_{\Gamma}\int_0^\delta \left(\sigma  \pi_r u_r+\mu\nabla_s \pi \nabla_s u \right) (1+2Hr+\kappa r^2) \\
    &-\int_0^T\int_{\Gamma}\int_0^\pi \left(\sigma \pi_r u_r+\mu\nabla_\Gamma \pi \nabla_\Gamma u \right) -\int_0^T\int_{\Gamma}\int_0^\delta (\sigma \pi_r u_r+\mu\nabla_\Gamma \pi \nabla_\Gamma u)(2Hr+\kappa r^2) \\
    & -\int_0^T\int_{\Gamma}\int_0^\delta \mu (\nabla_s \pi \nabla_s u-\nabla_\Gamma \pi \nabla_\Gamma u)(1+2Hr+\kappa r^2)\\
    =:&I+II+III.
\end{split} \tag{3.31}
\end{equation*}
As noted, write down
\begin{equation*}\label{eq332}
\begin{split}
      \mathcal{I}&= -\int_0^\delta \int_{\Gamma} (\sigma \pi_ru_r+\mu\nabla_\Gamma  \pi \nabla_\Gamma u) =\sqrt{\sigma\mu}\int_{\Gamma}u(p(s), t)  \Pi_R(s, 0).
\end{split}\tag{3.32}
\end{equation*}
Using the same estimates as in (\ref{eq56}) and (\ref{eq327}), we get
\begin{equation*}\label{eq333}
\begin{split}
    |II+III|    \leq&  O(\delta) \int_0^T \frac{1}{\sqrt{t}} \left(\int_\Gamma \sigma |\pi_r(s,0,t)| \right)^{1/2}dt
    =  O(\delta)  \sqrt{T}(\sigma\mu)^{1/4} ||\Pi_R||^{1/2}_{L^\infty(\Gamma)}.
\end{split}\tag{3.33}
\end{equation*}
Next, we consider the following cases $(1) \sigma\mu \to 0$, $(2) \sqrt{\sigma\mu}\to \gamma\in (0,\infty)$, $(3) \sigma \mu \to \infty.$

\medskip
\noindent
\textbf{Case $1$}. $\sigma\mu \to 0$ as $\delta\to 0$.  By (\ref{eq38}), (\ref{eq318}) and (\ref{eq319}), we have 
\begin{equation*}
    \begin{split}
      RHS &\leq O(1)\int_0^T\left(\int_0^\delta \int_{\Gamma} \sigma \pi_r^2+\mu|\nabla_\Gamma  \pi|^2\right)^{1/2}\left(\int_{\Omega} \nabla u\cdot A \nabla u\right)^{1/2} dt \leq O(1)  \sqrt{T}(\sigma\mu)^{1/4},
    \end{split}
\end{equation*}
where H\"older inequality and Lemma \ref{Estimate1} were used. So, we have $\mathcal{L}[v,\xi]=0$, implying $v$ satisfies $\frac{\partial v}{\partial \textbf{n}}=0$ on $\Gamma\times(0,T)$.

\medskip
\noindent
\textbf{Case $2$}. $\sqrt{\sigma\mu}\to \gamma\in (0, \infty)$ as $\delta \to 0$.

\noindent 
Subcase $(2i)$. $\mu\delta\to 0$ as $\delta \to 0$. In this case, $h \to 0$. In terms of (\ref{eq318}), (\ref{eq319}), (\ref{eq332}) and (\ref{eq333}), we have $I \to 0 $ and $|II+III| \to 0$, from which we have  $\mathcal{L}[v,\xi]=0$. So,  $v$ satisfies  $\frac{\partial v}{\partial \textbf{n}}=0$ on $\Gamma\times(0,T)$.

\medskip
\noindent Subcase $(2ii)$. $\mu\delta\to \beta\in (0,\infty]$ as $\delta \to 0$. In this case, $h \to H = \beta/\gamma \in(0, \infty]$. As $\delta \to 0$, it follows from (\ref{eq319}) and (\ref{eq333}) that
$I \to \gamma\int_0^T\int_{\Gamma} v\mathcal{J}_N^{\beta/\gamma}[\xi]$  and  $|II+III| \to 0,$
from which we get $\mathcal{L}[v,\xi]=\gamma\int_0^T\int_{\Gamma} v\mathcal{J}_N^{\beta/\gamma}[\xi].$ So, $v$ satisfies $k\frac{\partial v}{\partial \textbf{n}}=\gamma \mathcal{J}_N^{\beta/\gamma}[v]$ on $\Gamma\times(0,T)$.

\medskip
\noindent 
\textbf{Case $3$}. $\sigma\mu\to \infty$ as $\delta \to 0$.
 
 \noindent 
Subcase (3i). $\mu\delta\to \beta \in [0,\infty)$. In this case, $ h  \to 0$. By virtue of (\ref{eq318}) and (\ref{eq332}), it holds that
\begin{equation*}
\begin{split}
      I&=\mu \delta \int_0^T \int_{\Gamma} \left(\Delta_\Gamma \xi + O(h^2)\right)u  \longrightarrow \beta \int_0^T \int_{\Gamma}v\Delta_\Gamma \xi. \\
\end{split}
\end{equation*}
Additionally, by (\ref{eq318}) and (\ref{eq333}), $|II+III|\to 0$ as $\delta \to 0$. Consequently, we get
\begin{equation*}\label{bc}
    \mathcal{L}[v,\xi]=\beta \int_0^T \int_{\Gamma}v\Delta_\Gamma \xi. \tag{3.34}
\end{equation*}
%implying the boundary condition on \Gamma \times (0, T)\Gamma \times (0, T):

Our next task is to prove that $v$ is the weak solution of (\ref{EPDE}) with the boundary condition $k\frac{\partial v}{\partial \textbf{n}}=\beta \Delta_\Gamma v$ on $\Gamma \times (0, T)$. To this end, it remains to show $v \in L^2 \left((0,T); H^1(\Gamma)\right). $

We start by asserting that $\overline{v}$ is the unique weak solution of (\ref{EPDE}), which satisfies (\ref{bc}) as well. It suffices to prove $v=\overline{v}$. Now consider $v-\overline{v}$, without loss of generality, also denoted by $v$. We then points out that $v$ is the weak solution of (\ref{EPDE}) with $u_0=f=0$. In particular,  by Lemma \ref{Estimate1}, $v \in V^{1,0}_2 \left( \Omega_1 \times (0, T)\right) \cap W^{1,1}_2 \left( \Omega_1 \times (t_0, T)\right)$. 

For any small $t_0\in (0, T)$, fix $t_1\in (t_0, T]$. As $\delta \to 0$,  (\ref{eq322}) is transformed into
\begin{equation*}\label{eq335}
\begin{split}
    \int_{t_0}^{t_1} \int_{\Omega_1}(v_t\xi+k\nabla v \nabla \xi)dxdt&=\beta\int_{t_0}^{t_1} \int_{\Gamma}v\Delta_\Gamma \xi dsdt.\\
\end{split}\tag{3.35}
\end{equation*}
Furthermore, take the test function $\xi=w(s,t)\eta(r)$ with the following assumptions: $\eta=\eta(r)$ is a cut-off function in the $r$ variable with $0 \leq \eta\leq 1$, satisfying $\eta\in C^\infty(-\infty,0]$, $\eta=1$ for $-\epsilon \leq r\leq 0$ and $\eta=0$ for $r\leq -2\epsilon$; $w(s,t) \in C^2(\Gamma\times [0,T])$. From (\ref{eq335}), we are led to
\begin{equation*}\label{eq336}
    \begin{split}
        \beta \left|\int_{t_0}^{t_1} \int_{\Gamma}v\Delta_\Gamma \xi dsdt \right|&=\left|\int_{t_0}^{t_1} \int_{\Omega_1}(v_t\xi+k\nabla v \nabla \xi) dxdt\right | \leq C\|v\|_{W^{1,1}_2(\Omega_1\times (t_0,t_1))} \|w\|_{L^2((t_0,t_1);H^1(\Gamma))}.
    \end{split}\tag{3.36}
\end{equation*}
Consider such $w$ with  
$$\int_{t_0}^{t_1} \int_{\Gamma} w dsdt =0.
$$ 
We then define a linear functional $\mathcal{F}$: $w \to  \int_{t_0}^{t_1} \int_{\Gamma} v \Delta_{\Gamma} w dsdt,$ which is well-defined by $(\ref{eq336})$. This functional can be extended to the Hilbert space
\begin{equation*}
    \begin{split}
       \mathbb{H} = \{ w \in L^2\left((t_0,t_1);H^1(\Gamma)\right):  \int_{t_0}^{t_1}\int_{\Gamma} w  dsdt &=0\}
    \end{split}
\end{equation*}
with the inner product as $ \langle w_1, w_2 \rangle:= -\int_{t_0}^{t_1} \int_{\Gamma} \nabla_{\Gamma} w_1 \cdot \nabla_{\Gamma} w_2$. From Riesze representation theorem, there is some $z \in \mathbb{H} $ satisfying
\begin{equation*}\label{eq337}
    \begin{split}
       \mathcal{F}(w) &=-\int_{t_0}^{t_1} \int_{\Gamma}  \nabla_{\Gamma} z \cdot  \nabla_{\Gamma} w dsdt =\int_{t_0}^{t_1} \int_{\Gamma} z \Delta_{\Gamma} w dsdt.
    \end{split}\tag{3.37}
\end{equation*}
Consequently, it follows from (\ref{eq337}) that
$ \int_{t_0}^{t_1} \int_{\Gamma} (v-z) \Delta_{\Gamma} w = 0. $
By Riesze theorem again, this means that $v - z = m(t)$ for some function $m(t) \in \mathbb{H}$ and thus $ v \in L^2\left((0, T);H^1(\Gamma)\right)$.

 Going back to (\ref{eq335}), from Lemma \ref{Estimate1}, we have
\begin{equation*}
    \begin{split}
       \int_{\Omega_1} v^2 (x,t_1) dxdt \leq \int_{\Omega_1} v^2 (x,t_0) dxdt,
    \end{split} 
\end{equation*}
from which we are done by sending $t_0 \to 0$ for Subcase $( 3i)$.

\smallskip
\noindent Subcase (3ii). $\mu\delta\to \infty$ as $\delta\to 0$. In this case, $h\to H \in [0, \infty]$ after passing to a subsequence. If $H = 0$,  then divided both sides of the equation (\ref{eq58}) by $\mu\delta$ and sending $\delta\to 0$, we obtain
$\int_0^T \int_{\Gamma} v \Delta_\Gamma \xi  = 0$,
implying that $v(\cdot) = m(t) $ on $\Gamma$ for almost everywhere $t\in (0,T)$. 

If $H\in (0,\infty]$, then divided both sides of (\ref{eq58}) by $\sqrt{\sigma\mu}$ and sending $\delta\to 0$, we obtain $\int_0^T \int_{\Gamma} v \mathcal{J}_N^H [\xi] = 0$, implying that $v(\cdot) = m(t) $ on $\Gamma$ for almost everywhere $t\in (0,T)$. We further take a special test function $\xi=\xi(t)$ on $\Gamma$ and a constant extension in $\Omega_2$ such that $\overline{\xi}=\xi(t)$, resulting in
$\mathcal{L}[v,\xi]=0.$ So, $v$ satisfies
$\int_{\Gamma} \frac{\partial v}{\partial \textbf{n}}=0$ on $\Gamma\times (0, T)$. Therefore we accomplish the whole proof.
\end{proof}
We conclude this section by asking a natural question: what is the effective boundary condition if two eigenvalues of the coating in the tangent directions are not identical? That is to say, $A(x)$ has two different  eigenvalues in the  tangent directions. We answer this question by considering \textit{Type II} condition (\ref{case2}) in the next section.

%-----------------Results for Type II -----------------------%
\section{EBCs for \textit{Type II} condition }\label{sec: case2}
In this section we always assert that $\Gamma$ is a topological torus and $A(x)$ satisfies  \textit{Type II} condition (\ref{case2}). The aim of this section is to address EBCs on $\Gamma \times (0, T)$ as the thickness of the layer decreases to zero. 

With the aid of the curvilinear coordinates $(s, r)$, we choose a convenient local chart on $\Gamma$. For any $p_0 \in \Gamma$, the portion of $\Gamma$ near $p_0$ can be parameterized as $p = (s)$ with $p(0) = p_0$,
$$ \pmb{\tau}_1 = p_{s_1}  \text{~and~} \pmb{\tau}_2= p_{s_2}. 
$$
More precisely, let $ \Gamma := \Gamma_1 \times \Gamma_2$ with $p(s_1,0) \in \Gamma_1$ and $p(0, s_2)\in \Gamma_2$. In $\overline{\Omega}_2$, the explicit formula of $A(x)$ can be expressed as 
\begin{equation*}\label{eq340}
	A(x) = \sigma \textbf{n}(p)\otimes\textbf{n}(p) + \mu_1 \pmb{\tau}_1(p) \otimes \pmb{\tau}_1(p)
	+ \mu_2 \pmb{\tau}_2(p) \otimes \pmb{\tau}_2(p). 
\end{equation*}

\begin{theorem}\label{thm2}
Suppose that $\Gamma$ is a topological torus and $A(x)$ is given in (\ref{PDE1}) or (\ref{PDE2}) and satisfies (\ref{case2}). Let $u_0 \in L^2(\Omega)$ and $f \in L^2(Q_T)$ with functions being independent of $\delta$. Assume further that without loss of generality, $\mu_1 > \mu_2$. Moreover, $\sigma, \mu_1$, and $\mu_2$ satisfy the scaling relationships
\begin{equation*}
\begin{split}
   \lim_{\delta \to 0}\frac{\mu_2}{\mu_1} &= c \in [0, 1], \quad \lim_{\delta \to 0}\frac{\sigma}{\delta} = \alpha \in [0, 1],\\ \quad \lim_{\delta\to 0}\sigma\mu_i &=\gamma_i\in[0,\infty],  \quad
\lim_{\delta\to 0}\mu_i\delta=\beta_i\in[0,\infty], \quad i =1,2.
\end{split}
\end{equation*}
$(i)$ If $ c \in (0,1]$, then as $\delta \to 0$, $u\to v$ weakly in $W_2^{1,0}(\Omega_1 \times (0,T))$, strongly in $C([0,T];L^2(\Omega_1))$, where $v$ is the weak solution of (\ref{EPDE}) subject to the  effective boundary conditions listed in Table \ref{tb2}.

\smallskip
\noindent 
$(ii)$ If $ c= 0$ and $\underset{\delta \to 0}{\lim}\delta^2\mu_1/\mu_2 = 0$, then $u\to v$ weakly in $W_2^{1,0}(\Omega_1 \times (0, T))$, strongly in $C([0,T];L^2(\Omega_1))$, where $v$ is the weak solution of (\ref{EPDE}) subject to the  effective boundary conditions listed in Table \ref{tb3}. 
\end{theorem}

\begin{table}[!htbp]
    \centering
    \caption{Effective boundary conditions on $\Gamma \times (0, T)$ for $c\in (0, 1]$.} \label{tb2}
    
    EBCs on $\Gamma\times (0,T)$ for (\ref{PDE1}).\\
    
    \medskip
    
    \begin{tabular}{l|lcl}
       As $\delta\to 0$   
       & \quad $\frac{\sigma}{\delta}\to 0$ 
       &  $\frac{\sigma}{\delta}\to\alpha\in(0,\infty)$ 
       & \quad $\frac{\sigma}{\delta}\to\infty$\\
        \hline
        \hline
      $\sigma\mu_1\to 0$  
      & \quad $\frac{\partial v}{\partial \textbf{n}}=0$	
      &  $k\frac{\partial v}{\partial \textbf{n}}=-\alpha v$	
      &  \quad $v=0$\\
        \hline
        $\sqrt{\sigma\mu_1}\to \gamma_1 \in (0, \infty)$	
        &  $k\frac{\partial v}{\partial \textbf{n}}=\gamma_1 \mathcal{K}_D^\infty [v]$ 
        &  $k\frac{\partial v}{\partial \textbf{n}}=\gamma_1 \mathcal{K}_D^{\gamma_1/\alpha} [v]$	
        & \quad $v=0$\\
        \hline
         $\sigma\mu_1\to \infty$
         &  \makecell{ 
        $\nabla_{\Gamma} v =0$,\\ $\int_{\Gamma}\frac{\partial v}{\partial\textbf{n}}=0 $ 
        }	& 
        \quad \makecell{$\nabla_{\Gamma} v =0$,\\ $\int_{\Gamma}(k\frac{\partial v}{\partial\textbf{n}}+\alpha v) =0 $
        }	
        & \quad $v=0 $  \\
        \hline
    \end{tabular}

    \bigskip
    
    EBCs on $\Gamma\times (0,T)$ for (\ref{PDE2}).\\
    
    \medskip
    
    \begin{tabular}{l|lcc}
       As $\delta\to 0$ 
       &  $\mu_1\delta\to 0$ 
       &   $\mu_1\delta \to \beta_1 \in(0,\infty)$ 
       &  $\mu_1\delta\to \infty$\\
        \hline
        
        $\sigma\mu_1\to 0$ 	
        &  $\frac{\partial v}{\partial \textbf{n}}=0$	
        &   $\frac{\partial v}{\partial \textbf{n}}=0$	
        &  $\frac{\partial v}{\partial \textbf{n}}=0$\\
         \hline
        $\sqrt{\sigma\mu_1}\to \gamma_1 \in (0,\infty)$	
        & $\frac{\partial v}{\partial \textbf{n}}=0$	
        & $k\frac{\partial v}{\partial \textbf{n}}=\gamma_1 \mathcal{K}_N^{\beta_1/\gamma_1} [v]$	
        & $k\frac{\partial v}{\partial \textbf{n}}=\gamma_1 \mathcal{K}_N^\infty [v]$\\
        
         \hline
         $\sigma\mu_1\to \infty$	
         
         &  $\frac{\partial v}{\partial \textbf{n}}=0$	
         
         & $k\frac{\partial v}{\partial\textbf{n}} = \beta_1 \left( \frac{\partial^2 v}{\partial \pmb{\tau}_1^2} + c \frac{\partial^2 v}{\partial \pmb{\tau}_2^2} \right)$ 
         
         &  
    \makecell{$\nabla_{\Gamma} v =0$,\\$\int_{\Gamma}\frac{\partial v}{\partial\textbf{n}}=0$}\\
    
        \hline
    \end{tabular}
\end{table}

\begin{table}[!htbp]
    \centering
    
    \caption{Effective boundary conditions on $\Gamma \times (0, T)$ for $c = 0$.}\label{tb3} 
    
    EBCs on $\Gamma\times (0,T)$ for (\ref{PDE1}).\\
    
    \medskip
     
    \begin{tabular}{l|lll}
        As $\delta\to 0$   
        &    \quad $\frac{\sigma}{\delta}\to 0$ 
        &  $\frac{\sigma}{\delta}\to\alpha\in(0,\infty)$ &  $\frac{\sigma}{\delta}\to\infty$\\
        \hline
        \hline
         $\sigma\mu_1\to 0$
         & \quad $\frac{\partial v}{\partial \textbf{n}}=0$	& \quad $k\frac{\partial v}{\partial \textbf{n}}=-\alpha v$	
         &  $v=0$\\
        \hline
         $\sqrt{\sigma\mu_1}\to \gamma_1 \in (0, \infty)$	
         &  $k\frac{\partial v}{\partial \textbf{n}}=\gamma_1 \Lambda_D^\infty [v]$ 
         &  $k\frac{\partial v}{\partial \textbf{n}}=\gamma_1 \Lambda_D^{\gamma_1/\alpha} [v]$	
         & $v=0$\\
         
        \hline
         $\sigma\mu_1\to \infty$, $\sigma\mu_2\to 0$
         & \makecell{$\frac{\partial v}{\partial \pmb{\tau}_1}=0$,\\
         $\int_{\Gamma_1}\frac{\partial v}{\partial \textbf{n}}=0$}
        & \makecell{$\frac{\partial v}{\partial \pmb{\tau}_1}=0$,\\
         $\int_{\Gamma_1}\left(\frac{\partial v}{\partial \textbf{n}} + \alpha v\right)=0$}
        & $v=0$
        \\
        \hline
        \makecell{$\sigma\mu_1\to \infty$,\\
         $ \sqrt{\sigma\mu_2} \to \gamma_2 \in (0, \infty)$}	
         &  \makecell{$\frac{\partial v}{\partial \pmb{\tau}_1}=0$,\\
         $\int_{\Gamma_1} \left(k\frac{\partial v}{\partial \textbf{n}}-\gamma_2 \mathcal{D}^\infty_D[v] \right)=0$}
        & 
        \makecell{$\frac{\partial v}{\partial \pmb{\tau}_1}=0$,\\
         $\int_{\Gamma_1}\left(k\frac{\partial v}{\partial \textbf{n}}-\gamma_2 \mathcal{D}^{\gamma_2/\alpha}_D[v]\right)=0$}
        & $v=0$\\
        
        \hline
        \makecell{
         $\sigma\mu_1\to \infty$,
         $\sigma\mu_2 \to \infty$}	
         
         &  \makecell{ 
        $\nabla_{\Gamma} v =0$,\\ $\int_{\Gamma}\frac{\partial v}{\partial\textbf{n}}=0 $ 
        }
        & 
        \makecell{ 
        $\nabla_{\Gamma} v =0$,\\ $\int_{\Gamma}\frac{\partial v}{\partial\textbf{n}}=0 $ }
        & $v =0$\\
        \hline
    \end{tabular}
\end{table}   
\begin{table}    
    \centering
    EBCs on $\Gamma\times (0,T)$ for (\ref{PDE2}).\\
    
    \medskip
      
    \begin{tabular}{l|lcc}
       As $\delta\to 0$ 
       & \quad $\mu_1\delta\to 0$ 
       & \qquad$\mu_1 \delta \to \beta_1 \in(0,\infty)$ 
       &  \qquad $\mu_1\delta\to \infty$\\
        \hline
        \hline
        $\sigma\mu_1\to 0$ 	
        &\quad $\frac{\partial v}{\partial \textbf{n}}=0$	
        &$\frac{\partial v}{\partial \textbf{n}}=0$	
        & \qquad $\frac{\partial v}{\partial \textbf{n}}=0$\\
         \hline
        $\sqrt{\sigma\mu_1} \to \gamma_1 \in (0,\infty)$	
        &\quad $\frac{\partial v}{\partial \textbf{n}}=0$	
        &\quad \qquad$k\frac{\partial v}{\partial \textbf{n}}=\gamma_1 \Lambda_N^{\beta_1/\gamma_1} [v]$	
        &\qquad \qquad $k\frac{\partial v}{\partial \textbf{n}}=\gamma_1 \Lambda_N^\infty [v]$\\
         \hline
        $\sigma\mu_1\to \infty$
        &\quad $\frac{\partial v}{\partial \textbf{n}}=0$	
        & \qquad $k\frac{\partial v}{\partial\textbf{n}} = \beta_1 \frac{\partial^2 v}{\partial \pmb{\tau}_1^2}$ 
         &  \qquad \qquad \textit{see next table}\\
        \hline
    \end{tabular}
    
    \bigskip
    
    \medskip
      
    \begin{tabular}{c|lcc}
       \makecell{As $\mu_1\delta \to \infty$, $\sigma\mu_1\to \infty$}
       & $\mu_2\delta \to 0$ &\quad$\mu_2\delta \to \beta_2\in(0,\infty)$ 
       &  $\mu_2 \delta \to \infty$\\
        \hline
        \hline
        \quad $\sigma\mu_2\to 0$ 	
        & \makecell{$\frac{\partial v}{\partial \pmb{\tau}_1}=0$,\\
         $\int_{\Gamma_1}\frac{\partial v}{\partial \textbf{n}}=0$}	
        & \makecell{$\frac{\partial v}{\partial \pmb{\tau}_1}=0$,\\
         $\int_{\Gamma_1}\frac{\partial v}{\partial \textbf{n}}=0$}	
        & \makecell{$\frac{\partial v}{\partial \pmb{\tau}_1}=0$,\\
         $\int_{\Gamma_1}\frac{\partial v}{\partial \textbf{n}}=0$}\\
         \hline
        $\sqrt{\sigma\mu_2} \to \gamma_2 \in (0,\infty)$	
        &\makecell{$\frac{\partial v}{\partial \pmb{\tau}_1}=0$,\\
         $\int_{\Gamma_1}\frac{\partial v}{\partial \textbf{n}}=0$}		
        
        & \makecell{$\frac{\partial v}{\partial \pmb{\tau}_1}=0$,\\
         $\int_{\Gamma_1}\left(k\frac{\partial v}{\partial \textbf{n}}-\gamma_2 \mathcal{D}_{N}^{\beta_2/\gamma_2} [v]\right)=0$}	
        &\makecell{$\frac{\partial v}{\partial \pmb{\tau}_1}=0$,\\
         $\int_{\Gamma_1}\left(k\frac{\partial v}{\partial \textbf{n}}-\gamma_2 \mathcal{D}_N^{\infty} [v]\right)=0$}\\
         \hline
         $\sigma\mu_2\to \infty$	
         &\makecell{$\frac{\partial v}{\partial \pmb{\tau}_1}=0$,\\
         $\int_{\Gamma_1}\frac{\partial v}{\partial \textbf{n}}=0$}	
         & \makecell{$\frac{\partial v}{\partial \pmb{\tau}_1}=0$,\\
         $\int_{\Gamma_1}\left(k\frac{\partial v}{\partial\textbf{n}}-\beta_2 \frac{\partial^2 v}{\partial \pmb{\tau}_2}\right) = 0$}	
         & \makecell{ 
        $\nabla_{\Gamma} v =0$,\\ $\int_{\Gamma}\frac{\partial v}{\partial\textbf{n}}=0 $ 
        }\\
        \hline
    \end{tabular}
\end{table}
The boundary condition $\frac{\partial v}{\partial \pmb{\tau}_1}=0$ on $\Gamma \times (0,T)$ means that $v$ is a constant in $s_1$ on $\Gamma$, but it may depend on $s_2$ and $t$. The boundary condition
$k\frac{\partial v}{\partial\textbf{n}} = \beta_1 \left( \frac{\partial^2 v}{\partial \pmb{\tau}_1^2} + c \frac{\partial^2 v}{\partial \pmb{\tau}_2^2} \right)$ can be viewed as a second-order partial differential equation on $\Gamma$.

For $H \in (0, \infty]$ and smooth $g(s)$,
$\mathcal{K}_D^{H}$ and $\mathcal{K}_N^{H}$ in Table \ref{tb2} are defined by 
$\left(\mathcal{K}_D^H[g],\mathcal{K}_N^H[g]\right)(s) :=\left(\Psi_R(s,0),\Phi_R(s,0) \right)$,
where $\Psi$ and $\Phi$ are, respectively, bounded solutions of
\begin{equation*}
    \left\{
             \begin{array}{ll}
             \Psi_{RR}+\Psi_{s_1s_1}+ c \Psi_{s_2s_2}=0 , &  \Gamma \times(0, H), \\
               \Psi(s, 0)=g(s), &   \Psi(s,H)=0,
             \end{array}
  \right.
\left\{
             \begin{array}{ll}
             \Phi_{RR}+ \Phi_{s_1s_1}+ c\Phi_{s_2s_2}=0 , &  \Gamma \times(0,H), \\
               \Phi(s,0)=g(s), &   \Phi_R(s,H)=0.
             \end{array}
  \right.
\end{equation*}
$\Lambda_D^{H}$ and $\Lambda_N^{H}$ in Table \ref{tb3} are defined by $\left(\Lambda_D^H[g], \Lambda_N^H[g]\right)(s):= \left(\Psi^0_R(s,0),\Phi^0_R(s,0)\right)$,
where $\Psi^0$ and $\Phi^0$ are the bounded solutions of
\begin{equation*}
    \left\{
             \begin{array}{ll}
             \Psi^0_{RR}+\Psi^0_{s_1s_1}=0 , &  \Gamma \times(0, H) ,\\
               \Psi^0(s,0)=g(s), &   \Psi^0(s,H)=0,
             \end{array}
  \right.
\left\{
             \begin{array}{ll}
             \Phi^0_{RR}+ \Phi^0_{s_1s_1}=0 , & \Gamma \times (0, H) ,\\
               \Phi^0(s,0)=g(s), &   \Phi^0_R(s,H)=0.
             \end{array}
 \right.
\end{equation*}
Finally, $\mathcal{D}_D^{H}$ and $\mathcal{D}_N^{H}$ are defined by $\left(\mathcal{D}_D^H[g], \mathcal{D}_N^H[g]\right)(s_2):= \left(\Psi_R(s_2,0), \Phi_R(s_2, 0)\right)$,
where $\Psi(s_2, R)$ and $\Phi(s_2, R)$ are the bounded solutions of
\begin{equation*}
    \left\{
             \begin{array}{ll}
             \Psi_{RR}+\Psi_{s_2s_2}=0 , &  \Gamma_2 \times(0, H) ,\\
               \Psi(s_2, 0) =g(s_2), &   \Psi(s_2,H)=0,
             \end{array}
  \right.
\quad
\left\{
             \begin{array}{ll}
             \Phi_{RR}+ \Phi_{s_2s_2}=0 , &  \Gamma_2 \times (0, H) ,\\
               \Phi(s_2,0)=g(s_2), &   \Phi_R(s_2,H)=0.
             \end{array}
 \right.
\end{equation*}

It is worth mentioning that the effective boundary conditions in Table \ref{tb2} are similar to those in \cite[Theorem 1]{CPW} and can be seen as a generalization in three dimensions, which implies that the smaller tangent thermal conductivity does not affect the boundary conditions if $ \lim_{\delta \to 0}\frac{\mu_2}{\mu_1} \neq 0$. However, under the assumption that $ \lim_{\delta \to 0} \frac{\mu_2}{\mu_1} = 0$, due to the smaller tangent thermal conductivity, new boundary conditions arise, which have not been encountered in the previous study (see the case of $\sqrt{\sigma \mu_2} \to \gamma_2 \in [0, \infty), \frac{\sigma}{\delta} \to \alpha \in [0, \infty)$ for \eqref{PDE1}, and $\sqrt{\sigma \mu_2} \to \gamma_2 \in [0, \infty], \mu_2\delta \to \beta_2 \in [0, \infty]$ for \eqref{PDE2}, as illustrated in Table \ref{tb3}).

\subsection{Definition, existence and uniqueness of weak solutions of effective models}
We define weak solutions of (\ref{EPDE}) together with some new boundary conditions from Table \ref{tb2} and \ref{tb3}.

\begin{definition}\label{def2}
Let the test function $\xi\in C^\infty (\overline{Q^1_T}) $ satisfy $\xi=0$ at $t=T$.

$(1)$  A function $v$ is said to be a weak solution of (\ref{EPDE}) with the boundary conditions $\frac{\partial v}{\partial \pmb{\tau}_1}=0$ and $\int_{\Gamma_1} \left(k\frac{\partial v}{\partial \textbf{n}}- B[v] \right) = 0$, where $B[v] = -\alpha v, \gamma_2\mathcal{D}^H_D[v]$ or $\gamma_2 \mathcal{D}^H_N[v]$ for $H \in (0, \infty]$,
if $v\in V^{1,0}_2(Q_T^1)$ and for almost everywhere fixed $t\in (0, T)$, the trace of $v$ on $\Gamma$ is a constant in $s_1$, and if for any test function $\xi$ satisfying $\frac{\partial \xi}{\partial \pmb{\tau}_1}=0$ on $\Gamma$, $v$ satisfies
$$\mathcal{L}[v,\xi] = \int_0^T\int_{\Gamma}v B[\xi] dsdt.
$$

$(2)$ A function $v$ is said to be a weak solution of (\ref{EPDE}) with the boundary condition $k\frac{\partial v}{\partial \textbf{n}}=\mathcal{B}[v]$, where $\mathcal{B}[v]= \gamma_1 \mathcal{K}_D^H[v] (\mathcal{K}_N^H[v])$, or $\gamma_1\Lambda_D^H[v] (\Lambda_N^H[v])$ for $H \in (0, \infty]$, if $v\in V^{1,0}_2(Q_T^1)$ and if for any test function $\xi$, $v$ satisfies
$$\mathcal{L}[v,\xi]=\int_0^T\int_{\Gamma} v \mathcal{B}[\xi] dsdt.
$$

$(3)$ A function $v$ is a weak solution of (\ref{EPDE}) with the boundary condition $k\frac{\partial v}{\partial\textbf{n}} = \beta_1 \left( \frac{\partial^2 v}{\partial \pmb{\tau}_1^2} + c \frac{\partial^2 v}{\partial \pmb{\tau}_2^2} \right)$ for $c \in [0,1]$, if $v\in V^{1,0}_2(Q_T^1)$ with its trace belonging to $L^2\left((0,T); H^1(\Gamma)\right)$, and if for any test function $\xi$, $v$ satisfies
$$
\mathcal{L}[v,\xi]=-\beta_1\int_0^T\int_{\Gamma} \left( \frac{\partial v}{\partial \pmb{\tau}_1} \frac{\partial \xi}{\partial \pmb{\tau}_1} + c \frac{\partial v}{\partial \pmb{\tau}_2}\frac{\partial \xi}{\partial \pmb{\tau}_2} \right) dsdt.
$$

$(4)$ A function $v$ is said to be a weak solution of (\ref{EPDE}) with the boundary conditions $\frac{\partial v}{\partial \pmb{\tau}_1}=0$ and $\int_{\Gamma_1} \left(k\frac{\partial v}{\partial\textbf{n}} - \beta_2\frac{\partial^2 v}{\partial \pmb{\tau}_2^2}\right)=0$, if $v\in V^{1,0}_2(Q_T^1)$ with its trace belonging to $L^2\left((0,T); H^1(\Gamma)\right)$ and being a constant in $s_1$, and if for any test function $\xi$ satisfying $\frac{\partial \xi}{\partial \pmb{\tau}_1}=0$ on $\Gamma$, $v$ satisfies
$$
\mathcal{L}[v,\xi]=-\beta_2\int_0^T\int_{\Gamma}  \frac{\partial v}{\partial \pmb{\tau}_2}\frac{\partial \xi}{\partial \pmb{\tau}_2}  dsdt.
$$
\end{definition}
Theorem \ref{Uni} also works for the existence and uniqueness of weak solutions of (\ref{EPDE}) together with above boundary conditions.

%-----------------Auxiliary functions-----------------------%
\subsection{Auxiliary functions }
We are now in a position to construct two auxiliary functions for \textit{Type II} condition (\ref{case2}). For every $t \in [0, T]$, let $\psi(s, r, t)$ and $\phi(s, r, t)$ be bounded solutions of 
\begin{equation*}\label{AF21}
\begin{split}
\left\{
    \begin{array}{ll}
      \sigma \psi_{rr}+\mu_1 \psi_{s_1s_1} + \mu_2 \psi_{s_2s_2}=0, & \Gamma \times(0, \delta),\\
    \psi(s,0,t) = g(s), &\psi(s, \delta, t) = 0,
    \end{array}
\right.
\end{split}\tag{4.1}
\end{equation*}
\begin{equation*}\label{AF22}
    \left\{
        \begin{array}{ll}
            \sigma \phi_{rr}+\mu_1 \phi_{s_1s_1}+ \mu_2\phi_{s_2s_2}=0, & \Gamma  \times(0, \delta),\\
            \phi(s,0,t) = g(s), & \phi_r(s, \delta, t)=0,
        \end{array}
  \right.\tag{4.2}
\end{equation*}
where $g(s) := \xi(s, 0, t)$. Let $r = R\sqrt{\sigma/\mu_1}$ and suppress the time dependence. Then, we define
$$\Psi^\delta(s, R):= \psi(s,R\sqrt{\sigma/\mu_1}, t), \quad \Phi^\delta(s, R) := \phi(s, R\sqrt{\sigma/\mu_1}, t).
$$ 
Plugging $r$ into (\ref{AF21}) and (\ref{AF22}) leads to
\begin{equation*}\label{eq43}
    \left\{
             \begin{array}{ll}
             \Psi^\delta_{RR}+ \Psi^\delta_{s_1s_1}+ \frac{\mu_2}{\mu_1}\Psi^\delta_{s_2s_2}=0 , &  \Gamma \times (0, h_1) ,\\
              \Psi^\delta (s ,0) = g(s), & \Psi^\delta (s, h_1) = 0,
             \end{array}
  \right.\tag{4.3}
\end{equation*}
\begin{equation*}\label{eq44}
\left\{
             \begin{array}{ll}
             \Phi^\delta_{RR}+\Phi^\delta_{s_1s_1}+ \frac{\mu_2}{\mu_1} \Phi^\delta_{s_2s_2}=0 , & \Gamma \times (0, h_1) ,\\
               \Phi^\delta(s, 0) = g(s), &   \Phi^\delta_R(s, h_1) = 0,
             \end{array}
  \right.\tag{4.4}
\end{equation*}
where $h_1 = \delta\sqrt{\sigma/\mu_1} $. We next estimate the size of $\Psi^\delta_R(s, 0)$ and $\Phi^\delta_R(s, 0)$ when the thickness of the thin layer is sufficiently small.

\smallskip
For  a fixed $ \delta >0$, rigorus formulas for $\Psi^\delta (s, R)$ and $\Phi^\delta (s, R)$ are expressed by separation of variables as follows
\begin{equation*}
	\begin{split}
		\Psi^\delta (s, R)
		=&-\sum_{n=1}^\infty \frac{\widetilde{g}^\delta_n \widetilde{e}^\delta_n(s)\left(e^{\sqrt{\widetilde{\lambda}^\delta_n}(R-h_1)}-e^{\sqrt{\widetilde{\lambda}^\delta_n}(h_1-R)}\right) }{2 \sinh\left(\sqrt{\widetilde{\lambda}^\delta_n}h_1\right)},
	\end{split}
\end{equation*}
\begin{equation*}
	\begin{split}
		\Phi^\delta (s, R)
		=&\sum_{n=1}^\infty \frac{\widetilde{g}^\delta_n \widetilde{e}^\delta_n(s)\left(e^{\sqrt{\widetilde{\lambda}^\delta_n}(R-h_1)}+e^{\sqrt{\widetilde{\lambda}^\delta_n}(h_1-R)}\right) }{2 \cosh\left(\sqrt{\widetilde{\lambda}^\delta_n}h_1\right)},
	\end{split}
\end{equation*}
where $\widetilde{\lambda}^\delta_n$ and $\widetilde{e}^\delta_n(s)$ are  the  eigenvalues and corresponding eigenfunctions of $-\widetilde{\Delta}^\delta_{\Gamma} := -\left(\frac{\partial^2}{\partial s_1^2} + \frac{\mu_2}{\mu_1}\frac{\partial^2}{\partial s_2^2}\right)$ defined on $\Gamma$ with $\widetilde{g}^\delta_{n} =  \langle \widetilde{e}^\delta_n, g \rangle := \int_{\Gamma} \widetilde{e}^\delta_n g ds.$ From this, a direct computation gives
\begin{equation*}\label{eqno45}
   \sqrt{\sigma\mu_1} \|     \Psi^\delta_R(s,0)\|_{L^\infty(\Gamma)}=\left\{
    \begin{aligned}
        &\frac{\sigma}{\delta}\left(-g(s)+O(h_1^2)\right),& & \text{ if }h_1 \to 0 \text{ as } \delta \to 0,\\
        & O(\sqrt{\sigma\mu_1}),& & \text{ if } h_1 \in (0,\infty] \text{ as } \delta \to 0,
    \end{aligned} 
\right. \tag{4.5}
\end{equation*}
\begin{equation*}\label{eq424}
   %\sigma \|\phi_r(s, 0, t)\|_{L^\infty(\mathbb{R}^2)}=
   \sqrt{\sigma\mu_1} \| \Phi^\delta_R(s, 0)\|_{L^\infty(\Gamma)}=\left\{
    \begin{aligned}
        &\mu_1\delta \left(\widetilde{\Delta}^\delta_\Gamma g(s)+O(h_1^2) \right), &  &\text{ if } h_1 \to 0 \text{ as } \delta \to 0,\\
        & O(\sqrt{\sigma\mu_1}),& & \text{ if } h_1\in (0,\infty] \text{ as } \delta \to 0.
    \end{aligned} 
\right. \tag{4.6}
\end{equation*}

Before diving further, we are intended to consider the limiting equation as $\delta \to 0$. In the case of $c \in (0,1]$, if $h_1 \to H\in(0,\infty)$ as $\delta \to 0$, then (\ref{eq43}) and  (\ref{eq44}) yield
\begin{equation*}\label{NLE1}
\left\{
             \begin{array}{ll}
             \Psi_{RR}+\Psi_{s_1s_1}+ c \Psi_{s_2s_2}=0 , &  \Gamma \times (0, H),\\
               \Psi(s,0)=g(s), &  \Psi(s, H)=0,
             \end{array}
\right. 
\left\{
             \begin{array}{ll}
             \Phi_{RR}+ \Phi_{s_1s_1}+ c \Phi_{s_2s_2}=0 , &  \Gamma \times (0, H),\\
               \Phi(s,0)=g(s), &   \Phi_R(s,H)=0.
             \end{array}
  \right.\tag{4.7}
\end{equation*}
It is easy to see that each of them has a unique bounded solution. We  define
$\left(\mathcal{K}_D^H[g], \mathcal{K}_N^H[g] \right)(s):= \left(\Psi_R(s,0), \Phi_R(s,0) \right). $ Furthermore,  their analytic formulas are given by
\begin{equation*}\label{eq48}
\begin{split}
  \mathcal{K}_D^H[g](s)
  =&-\sum_{n=1}^\infty \frac{\sqrt{\widetilde{\lambda}_n} \widetilde{e}_n(s)\widetilde{g}_n}{\tanh(\sqrt{\widetilde{\lambda}_n}H)}, \qquad
  \mathcal{K}_N^H[g](s) 
  =-\sum_{n=1}^\infty \sqrt{\widetilde{\lambda}_n} \widetilde{e}_n(s)\widetilde{g}_n \tanh(\sqrt{\widetilde{\lambda}_n}H),
\end{split}\tag{4.8}
\end{equation*}
where $\widetilde{\lambda}_n$ and $\widetilde{e}_n(s)$ are  the  eigenvalues and the corresponding eigenfunctions of $-\widetilde{\Delta}_{\Gamma} := -\left(\frac{\partial^2}{\partial s_1^2} + c \frac{\partial^2}{\partial s_2^2}\right)$ defined on $\Gamma$ with 
$\widetilde{g}_{n} = \langle \widetilde{e}_n, g \rangle:= \int_{\Gamma} \widetilde{e}_n g ds.$ Hence, $ \mathcal{K}_D^\infty[g](s) 
=\mathcal{K}_N^\infty[g](s) =-(-\widetilde{\Delta}_\Gamma)^{1/2}g(s).$

\smallskip
On the other hand, in the case of $c=0$, if $h\to H\in(0,\infty)$ as $\delta \to 0$, then it is easy to find that the limits of (\ref{eq43}) and (\ref{eq44}) are given in the following
\begin{equation*}\label{eq49}
    \left\{
             \begin{array}{ll}
             \Psi^0_{RR}+\Psi^0_{s_1s_1} =0 , &  \Gamma \times(0, H),\\
               \Psi^0(s,0)=g(s), &   \Psi^0(s,H)=0,
             \end{array}
   \right. 
\quad
\left\{
             \begin{array}{ll}
             \Phi^0_{RR}+ \Phi^0_{s_1s_1} =0, & \Gamma \times (0, H),\\
               \Phi^0(s,0)=g(s), &   \Phi^0_R(s,H)=0,
             \end{array}
  \right.\tag{4.9}
\end{equation*}
from which we  define 
$\Lambda_D^H[g](s) := \Psi^0_R(s,0)$ and  $\Lambda_N^H[g](s) := \Phi^0_R(s,0).$

\smallskip
The limiting equations in $(\ref{eq49})$ are degenerate and they are second-order equations with nonnegative characteristic form,  which is of interest in its own right. We refer the interested reader to the book \cite{OR1973} and the references therein. 

Because of the geometry of $\Gamma $, analytic formulas for $\Lambda_D^H[g](s)$ and $\Lambda_N^H[g](s)$ can be given using separation of variables. It is straightforward to see that
\begin{equation*}\label{eq410}
    \Lambda_D^H[g](s) = - S(s_2) \sum_{n=1}^{\infty}  \sqrt{\lambda^1_n} \coth\left(\sqrt{\lambda^1_n} H\right) g_n(s_2)e_n^1(s_1), \tag{4.10}
\end{equation*}
\begin{equation*}\label{eq411}
    \Lambda_N^H[g](s) = -  S(s_2) \sum_{n=1}^{\infty} \sqrt{\lambda^1_n} \tanh\left(\sqrt{\lambda^1_n} H\right) g_n(s_2)e_n^1(s_1), \tag{4.11}
\end{equation*}
where $\lambda^1_n$ and $e^1_n(s_1)$ are  the  eigenvalues and the corresponding eigenfunctions of $-\Delta_{\Gamma_1} := -\frac{\partial^2}{\partial s_1^2} $ defined on $\Gamma_1$ with 
$g_{n} (s_2)= \langle e^1_n, g \rangle := \int_{\Gamma_1} e^1_n g ds_1$; $S(s_2)$ is a continuous function with regard to $s_2 $. Moreover, $ \left( \Lambda_D^\infty[g], \Lambda_N^\infty[g] \right)(s) =\underset{ H\to \infty}{\lim}\left(  \Lambda_D^H[g],  \Lambda_N^H[g]\right)(s).$

It is worthwhile to mention that  for $H \in (0, \infty]$, $\left(\mathcal{K}^H_D[g], \mathcal{K}^H_N[g]\right)$ and $\left(\Lambda^H_D[g], \Lambda^H_N[g]\right)$ are well-defined for  any $g \in H^{\frac{1}{2}}(\Gamma)$. Furthermore, all these operators are  linear and symmetric.

\smallskip
From now on, we discuss the existence and uniqueness of the solution of (\ref{eq49}). From the maximum principle, it turns out that $\Psi^\delta$ and $\Phi^\delta$ are uniformly bounded and equicontinuous on  $\Gamma \times (0, H)$. Consequently, the Arzela-Ascoli compact theorem ensures that
\begin{equation*}
 \Psi^\delta \to \Psi^0, \quad \Phi^\delta \to \Phi^0
\end{equation*}
uniformly in $\Gamma \times (0, H)$ after passing to a subsequence of $\delta \to 0$; moreover, the limiting functions
\begin{equation*}
  \Psi^0 \in C\left( \Gamma \times (0, H) \right) \quad \text{ and } \quad \Phi^0 \in C\left( \Gamma \times (0, H) \right).
\end{equation*}
Our next task is to establish the uniqueness of $\Psi^0$ and $\Phi^0$. Following this goal, let $\Psi^0_1$ and $\Psi^0_2 $ be two solutions of the former equation in (\ref{eq410}); let $\Phi^0_1$ and $\Phi^0_2 $ be two solutions of the latter equation in (\ref{eq411}). Without loss of generality, consider $\Psi^0 = \Psi^0_1 - \Psi^0_2$ and $\Phi^0 = \Phi^0_1 - \Phi^0_2$, satisfying
\begin{equation*}\label{eq414}
    \left\{
             \begin{array}{ll}
             \Psi^0_{RR}+\Psi^0_{s_1s_1} =0 , & \Gamma \times(0, H),\\
               \Psi^0(s,0)=0, &   \Psi^0(s,H)=0;
             \end{array}
   \right. 
\left\{
             \begin{array}{ll}
             \Phi^0_{RR}+ \Phi^0_{s_1s_1} =0, & \Gamma \times (0, H),\\
               \Phi^0(s,0)=0, &   \Phi^0_R(s,H)=0.
             \end{array}
  \right.\tag{4.12}
\end{equation*}
Suppressing  the $s_2$ variable, letting $W (s_1, R):= \Psi^0(s_1,s_2,R) $ and  $V (s_1, R):= \Phi^0(s_1,s_2,R) $, we have
\begin{equation*}
	\left\{
	\begin{array}{ll}
		W_{RR}+W_{s_1s_1} =0 , & \Gamma_1  \times(0, H),\\
		W(s,0)=0, &   W(s,H)=0,
	\end{array}
	\right. 
	\left\{
	\begin{array}{ll}
		V_{RR}+ V_{s_1s_1} =0, & \Gamma_1 \times (0, H),\\
		V(s,0)=0, &   V_R(s,H)=0.
	\end{array}
	\right. \tag{4.13}
\end{equation*}
From the maximum principle, it turns out that $W=V=0$. Thus, the assertion of uniqueness of $\Psi^0$ and $\Phi^0$ is completed.

%----------------Main Proof--------------------------%
\subsection{Proof of Theorem \ref{thm2}}
The purpose of this subsection is to prove Theorem \ref{thm2} and address EBCs on $\Gamma \times (0, T)$.
\begin{proof}[The proof of Theorem \ref{thm2}]
By Theorem \ref{cpt}, given any subsequence of $\delta$, we can ensure that $u \to v$ weakly in $W^{1,0}_{2} \left(\Omega_1\times (0, T) \right)$, and strongly in $C\left( [0, T]; L^2(\Omega_1) \right)$  after passing to a further subsequence. In the following, we will prove that $v$ is a weak solution of (\ref{EPDE}) with boundary conditions listed in Table \ref{tb2} and \ref{tb3}. Because of the uniqueness as proved in Theorem \ref{Uni}, $u \to v$ without passing to any subsequence of $\delta >0$.

\smallskip
Based on the scaling relationships of $\sigma$, $\mu_1$ and $\mu_2 $ as $\delta \to 0$,  we obtain the effective boundary conditions for the Dirichlet probplem (\ref{PDE1}) and  the Neumann problem (\ref{PDE2}) respectively. 

\smallskip
\noindent 
\textbf {Step 1. Effective boundary conditions for the Dirichlet problem (\ref{PDE1}).}

\smallskip
Let $\xi\in C^{\infty}(\overline{\Omega}_1\times [0,T])$ with $\xi=0$ at $t=T$ and extend $\xi$ to the domain $\Omega \times (0, T)$ by defining
\begin{equation*}
  \overline{\xi}(x,t)=\left\{
       \begin{aligned}
    &\xi(x,t), & x\in \overline{\Omega}_1, \\
    &\psi(s(x), r(x), t), & x\in\Omega_2,
\end{aligned}
\right.
\end{equation*}
where $\psi$ is the solution of (\ref{AF21}) and it is easy to check $\overline{\xi}\in W^{1,1}_{2,0}(Q_T)$. 

By the weak convergence of $\{u\}_{\delta >0}$, as $\delta \to 0$, it follows from Definition \ref{def1} that 
\begin{equation*}\label{eq416}
    \begin{split}
       & \mathcal{L}[v, \xi] =-\underset{\delta \to 0}{\lim}\int_0^T\int_{\Omega_2}\nabla \psi \cdot A\nabla u dx dt.
      \end{split} \tag{4.14}
\end{equation*}
In the curvilinear coordinates $(s, r)$, the right-hand side of (\ref{eq416}) gives
\begin{equation*}\label{eq62}
    \begin{split}
     RHS := & -\int_0^T\int_{\Omega_2}\nabla \psi \cdot A\nabla u dx dt \\
     = &-\int_0^T\int_0^\delta\int_{\Gamma} (\sigma     \psi_r u_r+\nabla_\Gamma \psi \cdot A \nabla_\Gamma u) \\
     &-\int_0^T\int_0^\delta \int_{\Gamma} (\sigma     \psi_r u_r+\nabla_\Gamma \psi \cdot A \nabla_\Gamma u)(2H r+\kappa r^2)  \\
    & -\int_0^T\int_0^\delta \int_{\Gamma}(\nabla_s     \psi \cdot A \nabla_s u-\nabla_\Gamma \psi \cdot A \nabla_\Gamma u)(1+2Hr+\kappa r^2)\\
    =:&I+II+III.
    \end{split} \tag{4.15}
\end{equation*}
Multiplying (\ref{AF21}) by $u$ and performing integration by parts, we obtain
\begin{equation*}\label{eqno63}
  \begin{split}
     I :=\int_0^T \mathcal{I}dt& = \int_0^T\int_0^\delta \int_{\Gamma} \left(\sigma \psi_r u_r + \mu_1 \psi_{s_1} u_{s_1} +\mu_2 \psi_{s_2}u_{s_2}\right) dsdrdt\\
     & = -\int_0^T\int_{\Gamma} \sigma \psi_r(s,0,t) u(s,0,t)dsdt.
  \end{split}\tag{4.16}
\end{equation*}
Subsequently, it follows from (\ref{eq43}) and (\ref{eqno45}) that
\begin{equation*}\label{eqno64}
\begin{split}
    |II|  = & O(\delta) \int_0^T \left( \int_0^\delta \int_{\Gamma}  \sigma \psi_r^2 + \nabla_\Gamma \psi \cdot A \nabla_\Gamma \psi  \right)^{1/2}  \left( \int_0^\delta \int_{\Gamma} \sigma u_r^2 + \nabla_\Gamma u \cdot A \nabla_\Gamma u  \right)^{1/2}dt\\
    = & O(\delta) \int_0^T \frac{1}{\sqrt{t}}\left(\int_\Gamma \sigma |\psi_r(s,0,t)|ds \right)^{1/2}dt\\
    = & O(\delta)  \sqrt{T}(\sigma\mu_1)^{1/4} (||\Psi^\delta_R(s,0)||_{L^\infty(\Gamma)} )^{1/2},
\end{split}\tag{4.17}
\end{equation*}
where Lemma \ref{Estimate1} and H\"older inequality were used. 

\smallskip
To get the estimate for the third term $III$, using the curvilinear coordinates $(s, r)$, \eqref{case2} and \eqref{derivative}, we have 
\begin{equation*}
	\begin{split}
	\nabla_\Gamma u \cdot A \nabla_\Gamma \psi 
		=&\left(\sum_{k,l=1,2} g^{kl}(s,0) u_{s_k} p_{s_l} \right)\cdot \sum_{i,j=1,2} g^{ij}(s,0) \psi_{s_i}\mu_j p_{s_j} \\
		=&\sum_{i = 1, 2} \mu_i u_{s_i} \psi_{s_i} 
	\end{split}
\end{equation*}
and 
\begin{equation*}
	\begin{split}
	&\nabla_s u \cdot A \nabla_s \psi\\
	=&\left(\sum_{k,l=1,2} g^{kl}(s,r) u_{s_k} F_{s_l}(s,r) \right)\cdot \sum_{i,j=1,2} g^{ij}(s,r) \psi_{s_i}(\mu_j p_{s_j} + r A \textbf{n}_{s_j})\\
	=&\sum_{k,l=1,2}  \left( g^{kl}(s,0) + r g_r^{kl}(s, \overline{r}) \right) u_{s_k} (p_{s_l} + r \textbf{n}_{s_l}) \cdot \sum_{i,j=1,2} g^{ij}(s,r) \psi_{s_i}(\mu_j p_{s_j} + r A \textbf{n}_{s_j})\\
	= & \sum_{i = 1, 2} \mu_i u_{s_i} \psi_{s_i}  + O(r) \sum_{i, k= 1, 2} (\mu_i +\mu_k) u_{s_k} \psi_{s_i}  +O(r^2) (\mu_1 +\mu_2)\sum_{i, k= 1, 2}  u_{s_k} \psi_{s_i} 
	\end{split}
\end{equation*}
where we used Taylor expansion on $g^{ij}(s, r)$ and $\overline{r} \in (0, r)$. By virtue of these two formulas,  using \eqref{AF21} and Lemma \ref{Estimate1}, we get 
\begin{equation*}\label{eq420}
\begin{split}
    |III| =& O(1) \left| \int_0^T\int_0^\delta \int_{\Gamma} \nabla_s     \psi \cdot A \nabla_s u-\nabla_\Gamma \psi \cdot A \nabla_\Gamma u \right| \\
    = &  O(1)\left(\delta \sqrt{\frac{\mu_1}{\mu_2}} \right) \left(1+\delta \sqrt{\frac{\mu_1}{\mu_2}} \right) \int_0^T \int_0^\delta \int_{\Gamma} \left|\nabla_\Gamma \psi \cdot A \nabla_\Gamma u \right|\\
    = & O(1) \left(\delta \sqrt{\frac{\mu_1}{\mu_2}} \right) \left(1+\delta \sqrt{\frac{\mu_1}{\mu_2}} \right) \sqrt{T}(\sigma\mu_1)^{1/4} ||\Psi_R(s,0)||^{1/2}_{L^\infty(\Gamma)},
\end{split}\tag{4.18}
\end{equation*}
where we also used the following inequalities
$$
\sum_{i, k= 1, 2} (\mu_i +\mu_k) u_{s_k} \psi_{s_i}  \le C  \sqrt{\frac{\mu_1}{\mu_2}} \left| \sum_{i = 1, 2} \mu_i u_{s_i} \psi_{s_i}\right|
$$
and 
$$
(\mu_1+\mu_2)\sum_{i, k= 1, 2} u_{s_k} \psi_{s_i}  \le C  \frac{\mu_1}{\mu_2} \left| \sum_{i = 1, 2} \mu_i u_{s_i} \psi_{s_i}\right|.
$$

\smallskip

To investigate the asymptotic behavior of (\ref{eq62}) as $\delta \to 0$, we consider the following three cases 
$$
(1) \frac{\sigma}{\delta}\to 0, \quad (2) \frac{\sigma}{\delta}\to \alpha\in (0,\infty), \quad (3) \frac{\sigma}{\delta}\to \infty.
$$

\medskip
\noindent
\textbf{Case $1$}. $\frac{\sigma}{\delta}\to 0$ as $\delta \to 0$.

\noindent 
Subcase $(1i)$. $\sigma\mu_1\to 0$ as $\delta\to 0$. In view of (\ref{eq416}) - (\ref{eq420}) and H\"older inequality, we have
\begin{equation*}
    \begin{split}
      |RHS| \leq& O(1) \int_0^T\left(\int_0^\delta \int_{\Gamma} \left(\sigma \psi_r^2+\mu_1 \psi^2_{s_1s_1} + \mu_2\psi^2_{s_2s_2} \right)\right)^{1/2} dt
      \leq O(\sqrt{T}) \max\{ \sqrt{\frac{\sigma}{\delta}}, (\sigma\mu_1)^{1/4} \},
    \end{split}
\end{equation*}
where (\ref{eqno45}) and Lemma \ref{Estimate1} were used. Thus, we have $\mathcal{L}[v,\xi]=0$, showing that $v$ satisfies  $\frac{\partial v}{\partial \textbf{n}}=0$
on $\Gamma \times (0, T)$.

\smallskip
\noindent 
Subcase $(1ii)$. $\sqrt{\sigma\mu_1} \to \gamma_1 \in (0, \infty)$ as $\delta\to 0$. In this case, $h_1\to \infty$.  If $c \in (0, 1]$, then as $\delta\to 0$, it follows from (\ref{eq48}) and (\ref{eqno63}) that
\begin{equation*}
    \begin{split}
      \mathcal{I} & = \sqrt{\sigma\mu_1} \int_\Gamma \Psi^\delta_R(s, 0) u  \to \gamma_1 \int_{\Gamma} v\mathcal{K}_D^\infty[\xi].
    \end{split} \tag{4.19}
\end{equation*}
On the other hand, if $c = 0$, then from (\ref{eq410}) and (\ref{eqno63}), we have $\mathcal{I} \to \gamma_1 \int_{\Gamma} v \Lambda_D^\infty[\xi]$ as $\delta \to 0$.

Because of the assumption that $\delta\sqrt{\mu_1/\mu_2} \to 0$ as $\delta \to 0$, (\ref{eqno64}) and (\ref{eq420}) lead to $|II+III| \to 0 \text{ as } \delta\to 0.$
Hence, for $c \in (0, 1]$, we obtain
\begin{equation*}
    \mathcal{L}[v, \xi] = \gamma_1 \int_0^T \int_{\Gamma}v\mathcal{K}_D^\infty[\xi], \tag{4.20}
\end{equation*}
indicating that $v$ satisfies $k\frac{\partial v}{\partial \textbf{n}}=\gamma_1 \mathcal{K}_D^\infty[v]$ on $\Gamma \times (0,T)$; for $c = 0$, we obtain
$\mathcal{L}[v, \xi] = \gamma_1 \int_0^T \int_{\Gamma}v\Lambda_D^\infty[\xi], $
indicating that $v$ satisfies $k\frac{\partial v}{\partial \textbf{n}}=\gamma_1 \Lambda_D^\infty[v]$ on $\Gamma \times (0,T)$.

\smallskip
\noindent 
Subcase $(1iii)$. $\sigma\mu_1 \to \infty$ as $\delta\to 0$. In this case, $h_1 \to \infty$ as $\delta \to 0$. Divided both sides of (\ref{eq416}) by $\sqrt{\sigma\mu_1}$ and sending $\delta\to 0$, by (\ref{eqno63})-(\ref{eq420}), we are led to
\begin{equation*}\label{eq421}
\begin{split}
    \int_0^T\int_{\Gamma}v\mathcal{K}_D^\infty[g]=0 \text{ for } c\in (0, 1]  \quad \text{  and  } \quad  \int_0^T\int_{\Gamma}v\Lambda_D^\infty[g] = 0 \text{ for } c = 0.
\end{split}\tag{4.21}
\end{equation*}
In the case of $c\in (0,1]$, it holds from (\ref{eq48}) that $ \nabla_\Gamma v = 0$. By the similar proof in Subcase $(1iii)$ from Step 1 in the last section,  $v$ satisfies the boundary condition  $\int_{\Gamma}\frac{\partial v}{\partial \textbf{n}}=0$ on $\Gamma \times (0,T)$.

\smallskip
In the case of $c=0$, it follows from (\ref{eq410}) and (\ref{eq421}) that $v_{s_1} = 0$ on $\Gamma \times (0,T)$.  Next, choose the test function $\xi$ satisfying $\xi_{s_1} = 0$ on $\Gamma$. Let $\psi$ be a constant in $s_1$, and $\psi = \psi(s_2, r, t)$ is defined by 
\begin{equation*}\label{AF23}
    \left\{
        \begin{array}{ll}
             \sigma \psi_{rr}+ \mu_2\psi_{s_2s_2}=0, & \Gamma_2 \times(0, \delta),\\
              \psi(s_2, 0, t) = g(s_2), & \psi(s_2, \delta, t)=0,
        \end{array}
  \right.\tag{4.22}
\end{equation*}
where $g(s_2):= \xi(s_2, 0, t)$. The right-hand side of (\ref{eq416}) now depends on the relationships of $\delta$, $\sigma$ and $\mu_2$.

Continuing what we have done before, letting $r = R \sqrt{\sigma/\mu_2}$ and suppressing the $t$ dependence, we have $ \Psi^\delta(s_2, R): = \psi(s_2, R \sqrt{\sigma/\mu_2}, t)$. Substituting $r$ into (\ref{AF23}), we are led to
\begin{equation*}\label{eq69}
    \left\{
        \begin{array}{ll}
              \Psi^\delta_{RR}+ \Psi^\delta_{s_2s_2}=0, & \Gamma_2 \times(0, h_2),\\
              \Psi^\delta(s_2, 0) = g(s_2), & \Psi^\delta(s_2, h_2)=0,
        \end{array}
  \right.
 \longrightarrow 
  \left\{\begin{array}{ll}
  	\Psi_{RR}+ \Psi_{s_2s_2}=0, & \Gamma_2 \times(0, H),\\
  	\Psi(s_2, 0) = g(s_2), & \Psi^(s_2, H)=0,
  \end{array}
  \right. \tag{4.23}
\end{equation*}
where $h_2 = \delta \sqrt{\mu_2/\sigma}$. Moreover, we define $\mathcal{D}^H_{D}[g] (s_2):= \Psi_R(s_2, 0)$. We estimate the size of $\Psi^\delta_R(s_2, 0)$ as in (\ref{eqno45}), resulting in
\begin{equation*}\label{eq4160}
   \sqrt{\sigma\mu_2} \|     \Psi^\delta_R(s_2,0)\|_{L^\infty(\Gamma_2)}=\left\{
    \begin{aligned}
        &\frac{\sigma}{\delta}\left(-g(s_2)+O(h_2^2)\right),& & \text{ if }h_2 \to 0 \text{ as } \delta \to 0,\\
        & O(\sqrt{\sigma\mu_2}),& & \text{ if } h_2 \in (0,\infty] \text{ as } \delta \to 0.
    \end{aligned} 
\right. \tag{4.24}
\end{equation*}
If $\sigma\mu_2 \to 0$ as $\delta \to 0$, then by the similar argument in Subcase $(1i)$, we get $\mathcal{L}[v, \xi]  = 0$, showing that $v$ satisfies $ \int_{\Gamma_1}\frac{\partial v}{\partial \textbf{n}}=0$ on $\Gamma \times (0,T)$; if $\sqrt{\sigma\mu_2} \to \gamma_2 \in (0, \infty)$ as $\delta \to 0$, then by the similar argument in Subcase $(1ii)$, we get
\begin{equation*}
    \mathcal{L}[v, \xi] = \gamma_2 \int_0^T \int_{\Gamma}v\mathcal{D}_D^\infty[\xi], 
\end{equation*}
showing that $v$ satisfies $ \int_{\Gamma_1}\left(k\frac{\partial v}{\partial \textbf{n}} - \gamma_2\mathcal{D}_D^\infty[v]\right)=0$ on $\Gamma \times (0,T)$; if $\sigma\mu_2 \to \infty $ as $\delta \to 0$, then by the similar argument in Subcase $(1iii)$ from Section \ref{sec3.2}, we have
\begin{equation*}\label{eq4161}
\begin{split}
    \int_0^T\int_{\Gamma}v\mathcal{D}_D^\infty[g]=0,
\end{split}\tag{4.25}
\end{equation*}
which indicates that $ v_{s_2}=0$ on $ \Gamma\times(0, T)$. Thus, $v$ is a constant on $\Gamma$ in the spatial variable. Assume further that $\xi = m(t)$ on $\Gamma$ and $\psi(s, r, t) =\left(1 - r/\delta \right)m(t)$. Using the same technique in (\ref{eq58}), we get $\mathcal{L}[v,\xi]=0$ from which $v$ satisfies $\int_{\Gamma}\frac{\partial v}{\partial \textbf{n}}=0$ on $\Gamma \times (0,T).$

\medskip
\noindent 
\textbf{Case $2$}. $\frac{\sigma}{\delta}\to \alpha \in (0, \infty)$ as $\delta \to 0$.

\noindent 
Subcase $(2i)$. $\sigma\mu_1 \to 0$. In this case, $h_1 \to 0$. A combination of $(\ref{eq62})-(\ref{eq420})$ gives rise to
\begin{equation*}
    \mathcal{L}[v,\xi]=-\alpha\int_0^T\int_{\Gamma}v\xi,
\end{equation*}
from which $v$ satisfies the boundary condition $k\frac{\partial v}{\partial \textbf{n}}=-\alpha v$ on $\Gamma \times (0, T).$ 

\smallskip

\noindent 
Subcase $(2ii)$. $\sqrt{\sigma\mu_1} \to \gamma_1 \in (0, \infty)$. Like what we did in Subcase $(1ii)$, as $\delta \to 0$, if $c \in (0, 1]$, we have
\begin{equation*}
    \mathcal{L}[v, \xi] = \gamma_1 \int_0^T \int_{\Gamma}v\mathcal{K}_D^{\gamma_1/\alpha}[\xi], 
\end{equation*}
resulting in the boundary condition %on $\Gamma \times (0,T):
$k\frac{\partial v}{\partial \textbf{n}}=\gamma \mathcal{K}_D^{\gamma_1/\alpha}[v].$ 

On the other hand, if $c = 0$, we then have $\mathcal{L}[v, \xi] = \gamma_1 \int_0^T \int_{\Gamma}v\Lambda_D^{\gamma_1/\alpha}[\xi],$
resulting in the boundary condition $k\frac{\partial v}{\partial \textbf{n}}=\gamma_1 \Lambda_D^{\gamma_1/\alpha}[v]$ on $\Gamma \times (0,T)$.

\smallskip
\noindent 
Subcase $(2iii)$. $\sigma\mu_1 \to \infty$ as $\delta\to 0$. Following the proof of Subcase $(1iii)$, we are led to
\begin{equation*}
\begin{split}
 \int_0^T\int_{\Gamma}v\mathcal{K}_D^\infty[g]=0 \text{ for } c\in (0, 1] \quad \text{    and   }\quad
  \int_0^T\int_{\Gamma}v\Lambda_D^\infty[g] = 0, \text{ for } c =0.
\end{split}
\end{equation*}
Therefore, if $c \in (0, 1]$, then $\nabla v = 0$ on $\Gamma$, implying that $v$  satisfies $\int_{\Gamma}\frac{\partial v}{\partial \textbf{n}}=0$ on $\Gamma \times (0,T)$. 

On the other hand, if $c = 0$, then $v_{s_1} =0$. By further taking $\xi = \xi(s_2, r, t)$ and $\psi$ to be defined in (\ref{AF23}), performing the procedure in Subcase $(1iii)$, we arrive at the following results: if $\sigma\mu_2 \to 0$ as $\delta \to 0$, then
 $v$ satisfies $\int_{\Gamma_1}\frac{\partial v}{\partial \textbf{n}}=0$ on $\Gamma \times (0, T)$; if $\sqrt{\sigma\mu_2} \to \gamma_2 \in (0, \infty)$ as $\delta \to 0$, then  $v$ satisfies $\int_{\Gamma_1}\left(k\frac{\partial v}{\partial \textbf{n}} - \gamma_2\mathcal{D}_D^\infty[v]\right)=0;$
if $\sigma\mu_2 \to \infty$ as $\delta \to 0$, then  $v$ satisfies $ \nabla_\Gamma v = 0$ and
$\int_{\Gamma} \left(k\frac{\partial v}{\partial \textbf{n}}+\alpha v \right)=0$ on $\Gamma \times (0, T)$.

\medskip
\noindent
\textbf{Case $3$}. $\frac{\sigma}{\delta}\to \infty$ as $\delta \to 0$. 

\noindent 
Subcase $(3i)$. $\sqrt{\sigma\mu_1} \to \gamma_1 \in  [0, \infty)$. In this case, $h_1 \to 0$. In view of (\ref{eq62})- (\ref{eq420}), divided both sides of (\ref{eq416}) by $\sigma/\delta$ and sending $\delta \to 0$, we get
$\int_0^T\int_{\Gamma}v\xi=0$ from which $v$ satisfies $v=0$ on $\Gamma \times (0, T)$.

\medskip
\noindent 
Subcase $(3ii)$. $\sigma\mu_1 \to \infty$ as $\delta \to 0$. For the case of $c \in (0, 1]$, using the similar proof in Subcase $(3ii)$ in Section \ref{sec3.2}, we have the boundary condition $v =0$ on $\Gamma \times (0,T)$. 

On the other hand, for the case of $c =0$, after passing to a subsequence, we have $h_1 \to H\in[0,\infty]$ as $\delta\to 0$.  In view of (\ref{eq416})-(\ref{eq420}) and (\ref{eq410}), if $H =0$, then $v$ satisfies the boundary condition  $v=0$. Otherwise, if $H\in (0,\infty]$, we obtain
\begin{equation*}
 \int_0^T\int_{\Gamma}v \Lambda_D^H[\xi]=0,
\end{equation*}
showing that $v_{s_1} = 0$. Again, by taking $\xi = \xi(s_2, r, t)$ and $\psi$ defined in (\ref{AF23}), performing the procedure in Subcase $(1iii)$, we have $v = 0$ on $\Gamma \times (0, T)$.

\medskip
\noindent 
\textbf{Step 2. Effective boundary conditions for the Neumann problem (\ref{PDE2}). }

\smallskip
Let $\xi\in C^{\infty}(\overline{\Omega}_1\times [0,T])$ with $\xi=0$ at $t=T$. We extend $\xi$ to the domain $\Omega \times (0, T)$ by defining
\begin{equation*}
  \overline{\xi}(x,t)=\left\{
       \begin{aligned}
    &\xi(x,t), & x\in \overline{\Omega}_1, \\
  &\phi(s(x), r(x), t), & x\in\Omega_2,
\end{aligned}
\right.
\end{equation*}
where $\phi$ is the unique solution of (\ref{AF22}) and $\overline{\xi}\in W^{1,1}_{2}(Q_T)$.

Due to the weak convergence of $u \to v$ as $\delta \to 0$, it follows from Definition \ref{def} that
\begin{equation*}\label{eq426}
\begin{split}
    %\mathcal{L}[u,\overline{\xi}]=-\int_0^T\int_{\Omega_2}\nabla \phi \cdot A\nabla u dx dt
   % \longrightarrow
    \mathcal{L}[v,\xi] &= -\underset{\delta \to 0}{\lim}\int_0^T\int_{\Omega_2}\nabla \phi \cdot A\nabla u dx dt.
\end{split}\tag{4.26}
\end{equation*}
In the curvilinear coordinates $(s, r)$, rewrite the right-hand side of (\ref{eq426}) as
$$RHS:=- \int_0^T\int_{\Omega_2}\nabla \phi \cdot A\nabla u dx dt = I+II+III,
$$
where 
\begin{equation*}\label{eq427}
  \begin{split}
     I  := \int_0^T\mathcal{I} dt &= -\int_0^T\int_0^\delta \int_{\Gamma} \left(\sigma \phi_r u_r + \nabla_\Gamma \phi \cdot A \nabla u \right)ds dr dt\\
     &=-\int_0^T\int_0^\delta \int_\Gamma \sigma \phi_r(s,0,t) u ds dt,
  \end{split}\tag{4.27}
\end{equation*}
\begin{equation*}\label{eq428}
\begin{split}
    |II| = &\left|-\int_0^T\int_0^\delta \int_{\Gamma} (\sigma     \phi_r u_r+\nabla_\Gamma \phi \cdot A \nabla_\Gamma u)(2H r+\kappa r^2)\right|  \\
    =&  O(\delta)  \sqrt{T}(\sigma\mu_1)^{1/4} (||\Phi^\delta_R(s,0)||_{L^\infty(\Gamma)} )^{1/2},
\end{split}\tag{4.28}
\end{equation*}
\begin{equation*}\label{eq429}
\begin{split}
     |III| =& \left| -\int_0^T\int_0^\delta \int_{\Gamma}(\nabla_s     \phi \cdot A \nabla_s u-\nabla_\Gamma \phi \cdot A \nabla_\Gamma u)(1+2Hr+\kappa r^2)\right|\\
     = & O\left(\delta \sqrt{\frac{\mu_1}{\mu_2}} \right)\left(1+\delta \sqrt{\frac{\mu_1}{\mu_2}}\right)\sqrt{T}(\sigma\mu_1)^{1/4} ||\Phi^\delta_R(s,0)||^{1/2}_{L^\infty(\Gamma)}.
\end{split}\tag{4.29}
\end{equation*}

In the remainder of this section, we investigate the asymptotic behavior of $RHS$ as $\delta \to 0$ in the following cases
$(1)$ $ \sigma\mu_1\to 0$, $(2)$ $ \sqrt{\sigma\mu_1}\to \gamma \in (0,\infty)$, $(3)$ $ \sigma\mu_1\to \infty$.

\medskip
\noindent 
\textbf{Case $1$}.  $\sigma\mu_1\to 0$ as $\delta \to 0$.
Using H\"older inequality and (\ref{AF22}), we get
\begin{equation*}
    \begin{split}
      |RHS| \leq& O(1) \int_0^T\left(\int_0^\delta \int_\Gamma \sigma \phi_r^2+ \nabla_\Gamma \phi \cdot A \nabla_\Gamma \phi \right)^{1/2} \left( \int_{\Omega}\nabla u \cdot A\nabla u dx \right)^{1/2}dt = O(\sqrt{T}) (\sigma\mu_1)^{1/4},
    \end{split}
\end{equation*}
where (\ref{eq424}) and Lemma \ref{Estimate1} were used. Thus, we have $\mathcal{L}[v,\xi]=0$, showing that $v$ satisfies  $\frac{\partial v}{\partial \textbf{n}}=0$ on $\Gamma \times (0, T).$

\medskip
\noindent 
\textbf{Case $2$}. $\sqrt{\sigma\mu_1} \to \gamma_1 \in (0, \infty)$ as $\delta \to 0$.

\noindent 
Subcase $ (2i)$.  $\mu_1 \delta \to 0$ as $\delta \to 0$. In this case, $h_1 = \mu_1\delta/\sqrt{\sigma\mu_1} \to 0$. Thanks to (\ref{eq428})- (\ref{eq429}), we obtain
\begin{equation*}
    I \to 0 \quad \text{ and } \quad  |II+III| \to 0,
\end{equation*}
where (\ref{eq424}) was used. Thus, we have  $\mathcal{L}[v,\xi]=0$, showing that $v$ satisfies $\frac{\partial v}{\partial \textbf{n}}=0$ on $\Gamma\times(0,T)$.

\smallskip

\noindent 
Subcase $(2ii)$. $\mu_1 \delta \to \beta_1 \in (0, \infty ]$ as $\delta \to 0$. In this case, $h_1 \to \beta_1/ \gamma_1 \in (0, \infty]$. In view of (\ref{eq428})-(\ref{eq429}) and (\ref{eq424}), we are led to
\begin{equation*}
    \mathcal{L}[v,\xi]=\gamma_1\int_0^T\int_{\Gamma} v\mathcal{K}_N^{\beta_1/\gamma_1}[\xi] \text{ for } c \in (0, 1], \quad  \quad  \mathcal{L}[v, \xi]=\gamma_1\int_0^T\int_{\Gamma} v\Lambda_N^{\beta_1/\gamma_1}[\xi] \text{ for } c = 0,
\end{equation*}
from which $v$ satisfies $ k\frac{\partial v}{\partial \textbf{n}}=\gamma_1 \mathcal{K}_N^{\beta_1/\gamma_1}[v]$ for $c \in (0, 1]$ and
 $ k\frac{\partial v}{\partial \textbf{n}}=\gamma_1 \Lambda_N^{\beta_1/\gamma_1}[v]$ for $c=0$ on $\Gamma \times (0, T)$.

\medskip
\noindent 
\textbf{Case $3$}. $\sigma\mu_1 \to \infty$ as $\delta \to 0$.

\noindent Subcase $(3i)$. $\mu_1\delta\to \beta_1 \in [0, \infty)$. In this case, $h_1 \to \beta_1/ \gamma_1 \in [0, \infty)$. Combining (\ref{eq428})-(\ref{eq429}) and (\ref{eq424}), we have
\begin{equation*}
    \begin{split}
      I& \to \beta_1 \int_0^T \int_{\Gamma} \widetilde{\Delta}_{\Gamma} \xi(s, 0, t)v(s,0, t) \quad  \text{   and  } \quad  |II + III| \to 0,
    \end{split}
\end{equation*}
as $\delta \to 0$. Thus, for $c \in (0, 1]$, we arrive at $\mathcal{L}[v, \xi] = \beta_1 \int_0^T \int_{\Gamma} v \widetilde{\Delta}_\Gamma \xi,$
from which $v$ satisfies $k\frac{\partial v}{\partial \textbf{n}}=\beta_1 \widetilde{\Delta}_\Gamma v$ on $\Gamma \times (0, T)$; for $c = 0$, it holds that
$\mathcal{L}[v, \xi] = \beta_1 \int_0^T \int_{\Gamma} v \frac{\partial^2 \xi}{ \partial \pmb{\tau}_1^2},
$
from which $v$ satisfies $k\frac{\partial v}{\partial \textbf{n}}=\beta_1 \frac{\partial^2 v}{ \partial \pmb{\tau}_1^2} $ on $\Gamma \times (0, T)$.

\smallskip
\noindent 
Subcase $(3ii)$. $\mu_1\delta\to \infty$ as $\delta \to 0$. In this case, $h_1 \to H \in [0, H]$ after passing to a subsequence. We first consider the case of $c\in(0, 1]$.  If $H=0$,  we have
$$\int_0^T \int_{\Gamma} \widetilde{\Delta}_\Gamma \xi(p, 0, t)v dsdt = 0,$$
leading to $v(\cdot) = m(t) $ on $\Gamma$ for almost $t\in (0,T)$. If $H\in (0, \infty]$, then we have $\int_0^T \int_{\Gamma} v\mathcal{K}_N^H[\xi] = 0$, implying that $v(\cdot) = m(t) $ on $\Gamma$ for almost $t\in (0, T)$. Thus, we choose a special test function $\xi=\xi(t)$ on $\Gamma$ and a constant extension such that $\overline{\xi} (s, r, t)= \xi(t)$ in $\Omega_2$, indicating that we have $\mathcal{L}[v,\xi]=0$.  Hence, $v$ satisfies $\int_\Gamma \frac{\partial v}{\partial \textbf{n}}=0$ on $\Gamma \times (0, T)$.

For the case of $c=0$, if $H=0$,  we have
$$\int_0^T \int_{\Gamma} v\frac{\partial^2 \xi}{ \partial \pmb{\tau}_1^2} dsdt = 0,
$$
implying that $v(\cdot) = v(s_2, t) $ on $\Gamma$ for almost $t\in (0,T)$. If $H\in (0, \infty]$, it holds that $\int_0^T \int_{\Gamma} v\Lambda_N^H[\xi] = 0, $
implying that $v(\cdot) = v(s_2, t) $ on $\Gamma$ for almost $t\in (0,T)$. 

We start with the proof by taking a test function $\xi$ satisfying $\xi_{s_1} = 0$ on $\Gamma$. We further choose the auxiliary function $\phi = \phi(s_2, r, t)$ by defining
\begin{equation*}\label{AF24}
    \left\{
        \begin{array}{ll}
             \sigma \phi_{rr}+ \mu_2\phi_{s_2s_2}=0, & \Gamma_2 \times(0, \delta),\\
              \phi(s_2, 0, t) = g(s_2), &\phi_r(s_2, \delta, t)=0,
        \end{array}
  \right.\tag{4.30}
\end{equation*}
where $g(s_2):= \xi(s_2, 0, t)$. Letting $r = R \sqrt{\sigma/\mu_2}$ and suppressing the $t$ dependence, we have $ \Phi(^\delta s_2, R) = \phi(s_2, R \sqrt{\sigma/\mu_2}, t)$. Substituting these into (\ref{AF24}) gives 
\begin{equation*}\label{eq4167}
    \left\{
        \begin{array}{ll}
              \Phi^\delta_{RR}+ \Phi^\delta_{s_2s_2}=0, &\Gamma_2 \times(0, h_2),\\
              \Phi^\delta(s_2, 0) = g(s_2), &\Phi^\delta_R(s_2, h_2)=0,
        \end{array}
  \right.
   \longrightarrow
  \left\{ \begin{array}{ll}
  	\Phi_{RR}+ \Phi_{s_2s_2}=0, &\Gamma_2 \times(0, H),\\
  	\Phi(s_2, 0) = g(s_2), &\Phi_R(s_2, H)=0,
  \end{array}
 \right.\tag{4.31}
\end{equation*}
where $h_2= \delta \sqrt{\mu_2/\sigma}$. Moreover, we define $\mathcal{D}^{H}_{N}[g](s) := \Phi_R(s_2, 0)$. To estimate the size of $\Phi_R(s_2, 0)$ as in (\ref{eq424}), we have 
\begin{equation*}\label{eq4168}
   \sqrt{\sigma\mu_2} \|     \Psi^\delta_R(s_2,0)\|_{L^\infty(\Gamma_2)}=\left\{
    \begin{aligned}
        &\sqrt{\mu_2\delta}\left(\xi_{s_2s_2}(s_2, 0, t)+O(h_2^2)\right),& & \text{ if }h_2 \to 0 \text{ as } \delta \to 0,\\
        & O(\sqrt{\sigma\mu_2}),& & \text{ if } h_2\in (0,\infty] \text{ as } \delta \to 0.
    \end{aligned} 
\right. \tag{4.32}
\end{equation*}
From now on, we consider cases $(a)$  $\sigma\mu_2\to 0$, $(b$ $\sqrt{\sigma\mu_2}\to \gamma_2 \in (0,\infty)$, $(c)$ $ \sigma\mu_2\to \infty$.

\smallskip
\noindent 
Subcase $(3iia)$. $\sigma\mu_2\to 0$ as $\delta \to 0$. Like in Case $1$, we have $\mathcal{L}[v, \xi]=0$, showing that $v$ satisfies  $\int_{\Gamma_1}\frac{\partial v}{\partial \textbf{n}}=0$ on $\Gamma \times (0, T)$.

\smallskip

\noindent 
Subcase $(3iib)$. $\sqrt{\sigma\mu_2}\to \gamma_2 \in (0, \infty)$ as $\delta \to 0$. Assume further that $\mu_2\delta \to 0$. In this case, $h \to 0$ as $\delta \to 0$. By the similar proof as in Case 2, we have $\mathcal{L}[v, \xi] = 0$. So, $v$ satisfies the effective boundary condition $\int_{\Gamma_1}\frac{\partial v}{\partial \textbf{n}}=0$ on $\Gamma\times(0,T)$.

On the other hand, if $\mu_2 \delta \to \gamma_2 \in (0, \infty]$, then by (\ref{eq428})-(\ref{eq429}) and (\ref{eq424}), we are led to
\begin{equation*}
    \mathcal{L}[v, \xi] = \gamma_1\int_0^T\int_{\Gamma} v\mathcal{D}_N^{\beta_2/\gamma_2}[\xi],
\end{equation*}
from which $v$ satisfies  $ \int_{\Gamma_1} \left(k\frac{\partial v}{\partial \textbf{n}}-\gamma_1 \mathcal{K}_N^{\beta_2/\gamma_2}[v] \right)= 0$ on $\Gamma \times (0, T)$.

\smallskip

\noindent 
Subcase $(3iic)$. $ \sigma\mu_2 \to  \infty$ as $\delta \to 0$. Assume further that $\mu_2\delta \to \beta_2 \in [0, \infty)$. In this case, $h_2 \to 0$. By virtue of (\ref{eq428})-(\ref{eq429}) and (\ref{eq424}), we get
\begin{equation*}
    \mathcal{L}[v, \xi] = \beta_2\int_0^T\int_{\Gamma} v\frac{\partial^2 \xi}{\partial \pmb{\tau}^2_2},
\end{equation*}
from which $v$ satisfies $ \int_{\Gamma_1} \left(k\frac{\partial v}{\partial \textbf{n}}-\beta_2 \frac{\partial^2 v}{\partial \pmb{\tau}^2_2} \right) = 0$ on $\Gamma \times (0, T)$.

If $\mu_2 \delta \to \infty$, then in this case, $h_2 \to H \in [0, \infty]$ after passing to a subsequence. If $H = 0$,  then divided both sides of the equation (\ref{eq426}) by $\mu_2\delta$ and sending $\delta\to 0$, we obtain $\int_0^T \int_{\Gamma} v \frac{\partial^2 \xi}{\partial \pmb{\tau}^2_2}  = 0$, implying that $v(\cdot) = m(t) $ on $\Gamma$ for almost $t\in (0,T)$. 

If $H\in (0,\infty]$, then divided both sides of \eqref{eq426} by $\sqrt{\sigma\mu_2}$ and sending $\delta\to 0$, we obtain $\int_0^T \int_{\Gamma} v \mathcal{D}_N^H[\xi] = 0$, implying that $v(\cdot) = m(t) $ on $\Gamma$ for almost $t\in (0,T)$. 

Therefore, by taking a special test function $\xi = \xi(t)$ on $\Gamma$ and using a constant extension $\overline{\xi} = \xi(t)$ in $\overline{\Omega}_2$, we obtain $\mathcal{L}[v, \xi]=0,$ implying that $v$ the boundary condition $\int_{\Gamma} \frac{\partial v}{\partial \textbf{n}}=0$ on $\Gamma \times (0, T)$. The proof is thus  completed.
\end{proof}

%---------------------------Acknowledgments-----------------------------%
\section*{Acknowledgments}

The author is indebted to his advisor Professor Xuefeng Wang for his guidance and Dr. Yantao Wang for his helpful discussions. The author also thanks the anonymous referees for their helpful comments and suggestions, which have improved the manuscript.

% You may incorporate your references as follows in your main tex file.
% Using BibTex is not recommended but can be handled.

%-----------------Reference-------------------%

%\bibitem{LN04}
%\newblock Y. Lou and T. Nagylaki,
%\newblock Evolution of a semilinear parabolic system for migration and selection in population genetics,
%\newblock \emph{J. Differential Equations}, \textbf{204} (2004), 292--322.
%
%
%
%
%
%
%
%\bibitem{LN06}
%\newblock Y. Lou and T. Nagylaki,
%\newblock Evolution of a semilinear parabolic system for migration and selection without dominance,
%\newblock \emph{J. Differential Equations}, \textbf{225} (2006), 624--665.
%
%
%
%
%
\


\begin{thebibliography}{99}

\bibitem{BCF} 
    \newblock  H. Brezis, L.A. Caffarelli and A. Friedman
    \newblock \emph{Reinforcement problems for elliptic equations and variational inequalities},
    \newblock Ann. Mat. Pura Appl., \textbf{123} (1980),  219--246.

\bibitem{BK} 
    \newblock  G. Buttazzo and R. V. Kohn, 
    \newblock \emph{Reinforcement by a thin layer with oscillating thickness},
    \newblock Appl. Math. Optim., \textbf{16} (1987),  247--261.


\bibitem{HJ}
     \newblock  H. Carslaw and J. Jaeger,
     \newblock \emph{Conduction of heat in solids},
     \newblock  Reprint of the second edition, New York, 1988.


\bibitem{CPW} 
    \newblock  X. Chen, C. Pond and X. Wang, 
    \newblock \emph{Effective boundary conditions resulting from anisotropic and optimally aligned coatings: the two dimensional case},
    \newblock Arch. Ration. Mech. Anal., \textbf{206} (2012), 911–951.




\bibitem{CR1990} 
\newblock  P. Colli and J.-F. Rodrigues,
\newblock \emph{Diffusion through thin layers with high specific heat},
\newblock Asymptotic Anal., \textbf{3} (1990), 249–263.





\bibitem{GENG} 
\newblock X. Geng,
\newblock \emph{Effective boundary conditions arising from the heat equation with three-dimensional interior inclusion},
\newblock Comm. Pure Appl. Anal., \textbf{22} (2023), 1394-1419.



\bibitem{GLR2022} 
\newblock Y. Giga, M. Łasica and P. Rybka,
\newblock \emph{The heat equation with the dynamic boundary condition as a singular limit of problems degenerating at the boundary},
\newblock Asymptotic Analysis, \textbf{135} (2023), 463-508.








\bibitem{GT} 
    \newblock  D. Gilbarg and N. Trudinger,
    \newblock   \emph{Elliptic partial differential equations of second order},
    \newblock   Reprint of the 1998 edition, Springer-Verlag, Berlin, 2001. 



\bibitem{LW2022} 
    \newblock H. Li, J. Li and X. Wang,
    \newblock  \emph{Error estimates and lifespan of effective boundary conditions for 2-dimensional optimally aligned coatings},
    \newblock J. Differential Equations, \textbf{303} (2022), 1-41.



\bibitem{LW2017} 
    \newblock H. Li and X. Wang,
    \newblock \emph{Using effective boundary conditions to model fast diffusion on a road in a large field},
    \newblock Nonlinearity, \textbf{30} (2017), 3853--3894.


\bibitem{LW} 
    \newblock H. Li and X. Wang,
    \newblock \emph{Effective boundary conditions for the heat equation with interior inclusion},
    \newblock Commun. Math. Res., \textbf{36} (2020), 272--295.


\bibitem{L} 
    \newblock J. Li,
    \newblock \emph{Asymptotic behavior of solutions to elliptic equations in a coated body},
    \newblock Comm. Pure App. Anal., \textbf{8} (2009), 1251--1267.


\bibitem{LZRW2009}
    \newblock J. Li, S. Rosencrans, X. Wang and K. Zhang,
    \newblock \emph{Asymptotic analysis of a Dirichlet problem for the heat
              equation on a coated body},
    \newblock Proc. Amer. Math. Soc., \textbf{137} (2009), 1711-1721.


\bibitem{LSWW2021} 
    \newblock J. Li, L. Su, X. Wang and Y. Wang,
    \newblock \emph{Bulk-surface coupling: derivation of two models},
    \newblock J. Differential Equations, \textbf{289} (2021), 1-34.


    
\bibitem{LWZZ}
    \newblock J. Li, X. Wang, G. Zhang and K. Zhang,
    \newblock \emph{Asymptotic behavior of {R}obin problem for heat equation on a coated body},
    \newblock Rocky Mountain J. Math., \textbf{42} (2012), 937--958.


\bibitem{LZ} 
    \newblock J. Li and K. Zhang, 
    \newblock \emph{Reinforcement of the Poisson equation by a thin layer},
    \newblock Math. Models Methods Appl. Sci., \textbf{21} (2011), 1153--1192.



\bibitem{LV} 
    \newblock Y.Y. Li and M. Vogelius, 
    \newblock \emph{Gradient estimates for solutions to divergence form elliptic equations with discontinuous coefficients},
    \newblock Arch. Ration. Mech. Anal., \textbf{153} (2000),  91--151.


    
\bibitem{LM1972} 
    \newblock  J. L. Lions and E. Magenes,
    \newblock   \emph{Non-homogeneous boundary value problems and applications},
    \newblock   Springer-Verlag, New York, 1973.  
    
    
    
\bibitem{OR1973} 
    \newblock O. A. Oleinik and E. V. Radkevic, 
    \newblock \emph{Second-Order Equations With Nonnegative Characteristic Form},
    \newblock Springer New York, New York, 1973.



\bibitem{RW} 
    \newblock S. Rosencrans and X. Wang,
    \newblock \emph{Suppression of the {D}irichlet eigenvalues of a coated body},
    \newblock SIAM J. Appl. Math., \textbf{66} (2006), 1895--1916; Corrigendum, \emph{SIAM J. Appl. Math.}, \textbf{68} (2008), p1202.


    
\bibitem{SP} 
    \newblock  E. Sanchez-Palencia,
    \newblock \emph{Problèmes de perturbations liés aux phénomènes de conduction à travers des couches minces de grande résistivité. (French)},
    \newblock J. Math. Pures Appl., \textbf{53} (1974),  251–269.



\bibitem{XW} 
    \newblock  X. Wang,
    \newblock \emph{Effective boundary conditions of diffusion equations on domains containing thin layers (in Chinese)},
    \newblock Sci. Sin. Math., \textbf{46} (2016), 709-724.


    
\bibitem{W1987} 
    \newblock J. Wloka, 
    \newblock \emph{Partial differential equations},
    \newblock Cambridge University Press, Cambridge, 1987.

\end{thebibliography}
\end{document}